\documentclass[a4paper,12pt]{amsart}

\usepackage[utf8]{inputenc}
\usepackage[T1]{fontenc}
\usepackage[english]{babel}

\usepackage{mathtools}
\usepackage{amsfonts}
\usepackage{amssymb}
\usepackage{amsthm}
\usepackage{tikz-cd}

\usepackage{enumitem}

\usepackage[hidelinks]{hyperref}

\newcommand{\pp}{\mathrm}
\newcommand{\vali}[1]{\mathopen{]} #1 \mathclose{[}}
\newcommand{\valir}[1]{\mathopen{[} #1 \mathclose{[}}
\newcommand{\valil}[1]{\mathopen{]} #1 \mathclose{]}}

\newcommand{\LRA}{\Leftrightarrow}

\newcommand{\ad}{\mathrm{ad}}
\newcommand{\e}{\mathrm{e}}

\newcommand{\Id}{\mathrm{Id}}
\newcommand{\Hdim}{\mathcal{H}\text{-}\mathrm{dim}} %%
\newcommand{\pois}{\backslash}
\DeclareMathOperator{\diag}{diag}
\newcommand{\matriisi}[2]{\left [\begin{array}{#1} #2 \end{array} \right]}
 
\DeclarePairedDelimiter\braket{\langle}{\rangle}
\DeclarePairedDelimiter\norm{\lVert}{\rVert}

\newcounter{listoja_varten}
\newenvironment{lista}%
{	\begin{list}{(\roman{listoja_varten})}{\usecounter{listoja_varten}}	}%
	{	\end{list}				}

\theoremstyle{plain}
\newtheorem{innercustomthmp}{Proposition}
\newenvironment{customprop}[1]
{\renewcommand\theinnercustomthmp{#1}\innercustomthmp}
{\endinnercustomthmp}

\newtheorem{innercustomthmc}{Conjecture}
\newenvironment{customotak}[1]
{\renewcommand\theinnercustomthmc{#1}\innercustomthmc}
{\endinnercustomthmp}
\theoremstyle{definition}
\newtheorem{maar}{Definition}[section]
\newtheorem{Esim}[maar]{Example}
\newtheorem{Kysy}[maar]{Question}

\theoremstyle{plain}
\newtheorem{lause}[maar]{Theorem}
\newtheorem{lausea}{Theorem}
\newtheorem{prop}[maar]{Proposition}
\newtheorem{lemma}[maar]{Lemma}
\newtheorem{seur}[maar]{Corollary}
\newtheorem{fakta}[maar]{Fact}
\newtheorem*{lausei}{Theorem}

\theoremstyle{remark}
\newtheorem{huom}[maar]{Remark}

\newcommand\YT{\rule{0pt}{2.6ex}}       % Top strut
\newcommand\AT{\rule[-1.2ex]{0pt}{0pt}} % Bottom strut

\newcommand{\laina}[1]{``#1''}

\begin{document}
%%%%%%%%%%%%%%%%%%%%%%%%%%%%%%%%
%%%%%%%%%%%%%%%%%%%%%%%%%%%%%%%%

%\title{On metric classifications of Lie groups of low dimension}
%\title[Metric equivalences of nilpotent groups and Heintze groups]{Metric equivalences of nilpotent groups and Heintze groups applied to classifications in low dimension}
\title[Metric equivalences of Heintze groups]{Metric equivalences of Heintze groups and applications to classifications in low dimension}

\author[Kivioja]{Ville Kivioja}

\address[Kivioja and Nicolussi Golo]{
	Department of Mathematics and Statistics, University of Jyv\"askyl\"a, 40014 Jyv\"askyl\"a, Finland}

\author[Le Donne]{Enrico Le Donne}
\address[Le Donne]{Department of Mathematics, University of Fribourg, Fribourg CH-1700, Switzerland}

\author[Golo]{Sebastiano Nicolussi Golo}

\email{kivioja.ville@gmail.com}
\email{enrico.ledonne@gmail.com}

\email{sebastiano2.72@gmail.com}
%\address{	Department of Mathematics and Statistics, University of Jyv\"askyl\"a, 40014 Jyv\"askyl\"a, Finland}
\thanks{V.K.\ and
E.L.D.\ were partially supported 
by the European Research Council (ERC Starting Grant 713998 GeoMeG `\emph{Geometry of Metric Groups}'). E.L.D.\ and S.N.G.\ were partially supported by the Academy of Finland (grant288501 `\emph{Geometry of subRiemannian groups}' and by grant 322898 `\emph{Sub-Riemannian Geometry via Metric-geometry and Lie- group Theory}'). V.K.\ was also supported by Emil Aaltonen Foundation. 
  }
 
\newcommand{\piste}{\,.}
\newcommand{\pilkku}{\,,}

% notation
\newcommand{\reach}[3]{(#1,#2)^{(#3)}}
\newcommand{\ARe}{\alpha_\pp{sr}}
\newcommand{\AIm}{\alpha_\pp{si}}
\newcommand{\Anil}{\alpha_\pp{nil}}
\newcommand{\aaa}{{\alpha_0}}
\newcommand{\conedim}{\pp{conedim}}
\newcommand{\realshadow}{{\mathbb{R}}}


\renewcommand{\subjclassname}{%
 \textup{2010} Mathematics Subject Classification}

\date{\today}

\begin{abstract}
We approach the quasi-isometric classification questions on Lie groups by considering low dimensional cases and isometries alongside quasi-isometries.
First, we present some new results related to quasi-isometries between Heintze groups. Then we will see how these results together with the existing tools related to isometries can be applied to groups of dimension 4 and 5 in particular. Thus we take steps towards determining all the equivalence classes of groups up to isometry and quasi-isometry.
We completely solve the classification up to isometry for simply connected solvable groups in dimension 4, and for the subclass of groups of polynomial growth in dimension 5.
\end{abstract}

\maketitle
\tableofcontents

\addtocounter{section}{-1}

\section{Introduction}   		

This paper is a contribution to various metric classifications of Lie groups. 
The study of quasi-isometries between solvable groups is an active area of research 
\cite{pansu,Shalom,  Sauer2006,Cornulier:qihlc,MR2925383,MR3034290,MR2594617,MR3180486,MR3335254,avain:CPS,avain:carrasco-orlicz,avain:Pal20}. 
Distinguished examples of solvable groups are Heintze groups, i.e., those solvable simply connected Lie groups that admit left-invariant Riemannian structures with negative sectional curvatures \cite{avain:He74}.
Every Heintze group \( G \) is a semidirect product of $\mathbb{R}$ and a nilpotent graded Lie group \(N\). 
The parabolic visual boundary of \( G \) has a structure of homogeneous group.
Namely, the boundary may be identified with \(N\) equipped with a distance that has dilation properties. Moreover, a quasi-isometry between two Heintze groups induces a quasisymmetry between the associated nilpotent groups and vice versa \cite{paulin_bord_hyp, Bonk-Schramm,Cornulier:qihlc}. These quasisymmetries are, or induce, biLipschitz maps between the boundaries equipped with suitable homogeneous structures \cite{avain:LD-Xie_fibers, avain:carrasco-orlicz}.

The main aim of this article is twofold: First, we introduce quasi-isometry invariants that finally distinguish some low dimensional Heintze groups. Second, 
we study a finer metric classification.
We say that two Lie groups \( G\) and \( H\) \emph{can be made isometric} if there are left-invariant Riemannian metrics \( \rho_G \) and \( \rho_H \) so that \( (G,\rho_G) \) is isometric  to \( (H,\rho_H) \). This is an equivalence relation among simply connected solvable groups, and we find the equivalence classes in low dimension: 
%In dimension we are able to find the equivalnce classes for all simply connected solvable Lie groups. In dimension 5 we treat.
we consider all simply connected solvable Lie groups in dimension 4 and those with polynomial growth in dimension \nolinebreak 5.
%in low dimension.
For each equivalence class, there is a Riemannian manifold for which each element of the class acts isometrically and simply transitively.
In our construction, such a Riemannian manifold is a Lie group, which we call the \laina{real-shadow}. 
In particular, we make a contribution to the conjecture that claims that every two Heintze groups are  either not quasi-isometric or 
they can be made isometric.

\subsection{Quasi-isometries of Heintze groups}

First we present our results related to distinguishing Heintze groups up to quasi-isometry equivalence. We work on the level of parabolic visual boundaries, thus our objects of interest are \emph{homogeneous groups}, by which we mean pairs \( (N,\alpha) \) where \( N \) is a simply connected nilpotent Lie group and \( \alpha \) is a derivation of \( N\), such that \( N \rtimes_\alpha \mathbb{R} \) defines a Heintze group. We may assume that \( N \rtimes_\alpha \mathbb{R} \) is purely real, i.e., that all the eigenvalues of \( \alpha \) are real numbers. 
For a homogeneous group \( (N,\alpha) \), we always consider the biLipschitz class of distances that are homogeneous under the one-parameter subgroups of automorphisms induced by the derivation \( \alpha \). This class may be empty in some cases, see Remark \ref{huom:existence_of_distance}.

Below we use the following notation: If \( (N,\alpha) \) is a homogeneous group and \( \bigoplus_{\lambda >0 } V^\alpha_\lambda \) is the decomposition of the Lie algebra of \( N \) by the generalised eigenspaces of the derivation \( \alpha \), then for every \( s > 0 \) we denote by \( \reach{N}{\alpha}{s} \) the subgroup of \( N\) with the Lie algebra \( \pp{LieSpan} (\bigoplus_{s \ge \lambda >0 } V^\alpha_\lambda) \).

\begin{lausea} \label{thm:computable-reachability-sets2}
	Let \( (N_1,\alpha) \) and \( (N_2,\beta) \) be purely real homogeneous groups that are biLipschitz equivalent via a map \( F \colon N_1 \to N_2 \). 
	\begin{lista}
		\item \label{thm-item-1} Then \( N_1 \) and \( N_2 \) are quasi-isometric as Riemannian Lie groups.
		\item \label{thm-item-2} For every \( p \in N_1 \) and every \( s \ge 1 \) we have \( F(p \reach{N_1}{\alpha}{s}) = F(p) \reach{N_2}{\beta}{s} \) and the same holds for all the iterated normalisers of the subgroups \( \reach{N_1}{\alpha}{s} \) and \( \reach{N_2}{\beta}{s} \), respectively.
	\end{lista}
\end{lausea}

The proof of this result, which is inspired by the results of \cite{avain:CPS}, is presented in Section \ref{sec:lemma_of_Seba}. We also present some examples to illustrate how this result helps to distinguish some particular pairs of low dimensional Heintze groups up to quasi-isometry. Notice that part (i) implies via \cite{pansu} that the Carnot groups associated to \( N_1\) and \( N_2\) as their asymptotic cones are isomorphic. In particular, the nilpotency steps of \( N_1\) and \( N_2\) agree.

\subsection{On the classification up to isometries}

To motivate and give some background, let us compare the state of the art of the classification up to isometry and quasi-isometry for two distinct subclasses of the class of solvable simply connected Lie groups: Heintze groups and solvable groups of polynomial growth (with nilpotent groups as main examples).
These subclasses have some similarities when it comes to isometries and quasi-isometries. In both cases every group has \laina{a representative with real roots} and those representatives are known to be distinguished by isometries, and are conjectured to be distinguished by quasi-isometries. More precisely, we have the following facts and folklore conjectures:

\begin{customprop}{H1}[Alekseevski\u\i  \,\cite{alekseevski}] \label{prop:alekseevski}
	Every Heintze group can be made isometric to a purely real Heintze group.
\end{customprop} 

\begin{customprop}{H2}[Gordon--Wilson \cite{avain:GW88}] \label{prop:Gordon-Wilson-Heintze}
	If two purely real Heintze groups can be made isometric, then they are isomorphic.
\end{customprop}

\begin{customotak}{H3} \label{otak:heintze-real}
		If two purely real Heintze groups are quasi-isometric then they are isomorphic.
\end{customotak}

\begin{customprop}{P1}[Breuillard \cite{avain:Bre14}]
	Every simply connected solvable Lie group of polynomial growth can be made isometric to a nilpotent group.
\end{customprop} 

\begin{customprop}{P2}[Wolf \cite{avain:Wol63}]
	If two simply connected nilpotent Lie groups can be made isometric, then they are isomorphic.
\end{customprop}

\begin{customotak}{P3} \label{otak:nilpotent}
	If two simply connected nilpotent Lie groups are quasi-isometric then they are isomorphic.
\end{customotak}

Recently, the articles \cite{avain:CKLNO} and \cite{avain:jablonski} clarified quite a bit this field, when it comes to isometries. Now we know that to every  simply connected solvable Lie group it is possible to canonically associate a completely solvable (a.k.a. split-solvable or real triangulable) Lie group, so called \laina{real-shadow} of the group, which is unique up to isomorphism. In particular, this construction satisfies the following theorem.

\begin{fakta}[Corollary 4.23 in \cite{avain:CKLNO}] \label{thm:graaf-CKLNO}
	Let \( G \) and \( H \) be simply connected solvable Lie groups.
	Then \( G \) can be made isometric to \( H \) if and only if the real-shadows of \(G\) and \( H\) are isomorphic.
\end{fakta}

This result, besides containing the information of propositions H1-2 and P1-2 above, implies that \laina{can be made isometric} is an equivalence relation within the class of simply connected solvable Lie groups. Moreover, it implies that the isometric classification of such groups boils down to the algebraic problem of calculating their real-shadows. Remark that it is not known if Fact \ref{thm:graaf-CKLNO} holds when isometries are replaced by quasi-isometries. This is because the more general version, due to Y.\ Cornulier \cite[Conjecture 19.113]{Cornulier:qihlc}, of Conjecture \ref{otak:heintze-real} and Conjecture \ref{otak:nilpotent} is also open: whether two quasi-isometric completely solvable simply connected Lie groups are necessarily isomorphic or not.

Since Lie groups that can be made isometric are necessarily quasi-isometric, we are led to study the following problem: Which pairs of groups in the same quasi-isometry class can be made isometric? 
This problem is completely solved for groups of dimension 3 and it is surveyed in \cite{avain:fasslerledonne}.
One of the main contributions of the present article is to push towards a solution for simply connected groups of dimension 4. While we are not able to completely solve the quasi-isometry relations of 4-dimensional groups, we can solve the isometry relations: it is clear that Fact \ref{thm:graaf-CKLNO} is enough for that. However, we will also prove Theorem \ref{thm:alkup_yleinen} below, which is a more explicit result and can be proved with elementary methods. In its statement, we denote by \( \alpha_0 =\ARe +  \Anil  \) the \emph{real part} of the derivation \( \alpha \): we shall recall the relevant decomposition more precisely in Proposition \ref{prop:decomposition_of_a_derivation}.

\begin{lausea} \label{thm:alkup_yleinen}
	Let \( H \) be a simply connected Lie group and \( \alpha \) a derivation of \( H\). 
	Then the Lie group \(  H \rtimes_{\alpha} \mathbb{R} \) can be made isometric to the Lie group \(  H \rtimes_{\alpha_0} \mathbb{R} \), where \( \alpha_0 \) is the real part of \( \alpha \).
\end{lausea}

In the category of solvable groups, the above result is a special case of Fact \ref{thm:graaf-CKLNO}, but it may also provide information about isometry questions of non-solvable semidirect products. Notice that there is no assumptions on the eigenvalues of \( \alpha \).

Theorem \ref{thm:alkup_yleinen} has practical value also within the family of solvable Lie groups:
In Section \ref{sec:4D} we find all the pairs of Lie groups that can be made isometric within the family of 4-dimensional simply connected solvable Lie groups.
In Section \ref{sec:5D} we do the same within the family of 5-dimensional simply connected solvable Lie groups of polynomial growth.
The method is described as follows. Since the algebraic classification of Lie groups is known within these families, we first indicate all the completely solvable ones: these are the groups that are isomorphic to their real-shadows. Then for each solvable group  \( G\) that is not completely solvable, we find a completely solvable group to which it is isometric by finding a suitable decomposition of \( G\) as a semi-direct product \( H \rtimes_\alpha \mathbb{R} \) where \( H \) is completely solvable. This happens to be always possible within the families we investigate. Now we know from Theorem \ref{thm:alkup_yleinen} that such a group \( G \) can be made isometric to the completely solvable group \( H \rtimes_{\alpha_0} \mathbb{R} \), while Fact \ref{thm:graaf-CKLNO} then guarantees that \( H \rtimes_{\alpha_0} \mathbb{R} \) is the real-shadow of \( G \), and any other solvable group \( G' \) that can be made isometric to \( G \) must also have \( H \rtimes_{\alpha_0} \mathbb{R} \) as the real-shadow.

The result we get in dimension 4 is summarised in the theorem below. 

\begin{lausea} \label{thm:4D-classification_intro}
	Let \( G \) and \( H \) be simply connected solvable Lie groups of dimension 4. 
	If \( G \) and \( H \) are both completely solvable, then they can be made isometric if and only if they are isomorphic. Instead, if at least one of them is not completely solvable, then they can be made isometric if and only if they belong to the same set of groups in the following list (the notation is w.r.t.\ the classification given by \cite{avain:Patera}):
	\begin{lista}
		\item[\( (\pp{I}) \)] \( \{  \mathbb{R}^4,\: \mathbb{R} \times A_{3,6}  \} \),
		\item[\( (\pp{II}) \)] \( \{  \mathbb{R} \times A_{3,1},\: A_{4,10} \} \),
		\item[\( (\pp{III}_{\lambda}) \)] \( \{ A_{4,5}^{\lambda,\lambda} \} \cup \{ A_{4,6}^{a,b} : \lambda = \mathrm{sign}(ab) \min(|b/a|,|a/b|) \} \),
		\item[\( (\pp{IV}) \)]  \( \{ A_{4,9}^1\} \cup \{A_{4,11}^{a} : a \in \vali{0,\infty}\, \}\),
		\item[\( (\pp{V}) \)] \( \{\mathbb{R} \times A_{3,3},\: A_{4,12} \} \cup \{ \mathbb{R} \times A_{3,7}^a  : a \in \vali{0,\infty}\, \}\),
%		\item[\( (\pp{VI}) \)] \( \{ \mathbb{R} \times A_{3,3},\: A_{4,12}  \}\),
		\item[\( (\pp{VI}) \)] \(  \{ \mathbb{R}^2 \times  A_2\} \cup \{ A_{4,6}^{a,0} : a \in \mathbb{R} \} \)
	\end{lista}
Here \( (\pp{III}_{\lambda}) \) stands for distinct sets depending on parameter \( \lambda \in \mathbb{R} \pois \{0\} \). Hence the above list contains 5 sets (2 finite and 3 infinite) and one family of sets depending on a parameter.
\end{lausea}
In Section \ref{sec:5D} we find %\footnote{this sentence repeats from the above, is it bad?}
 similar classification for simply connected solvable groups of polynomial growth in dimension 5. Table~\ref{table:5d-poly-growth} in Section \ref{sec:5D} summarises the results within this family. 

%\puna{Put some words of what we do in Section \ref{sec:5D} and how?}

\subsection*{Acknowledgements}
We thank Pierre Pansu and Matias Carrasco Piaggio for constructive comments. % on an early version of this article. 
In particular, Pansu helped us to strengthen an early version of
 Theorem~\ref{thm:computable-reachability-sets2}.(i).

\section{Preliminaries}

\label{sec:preliminaries}

\subsection{Homogeneous groups}
\label{sec:homog_groups}

%As we shall mainly\footnote{or only?} consider the quasi-isometry classification problems in Heintze groups by studying their boundary, we define the terminology of this setting next.
We shall approach the quasi-isometric classification problems in Heintze groups by studying the biLipschitz maps on their boundary, and we now define precisely the terminology of this setting.

In this paper, we will always use the convention that if \( N,H,G, \ldots \) are Lie groups, then the fraktur letters \( \mathfrak{n}, \mathfrak{h}, \mathfrak{g}, \ldots \) denote their Lie algebras, and vice versa.

%\sini{Explain the subalgebra \( H_\alpha \) from CPS.}

%\emph{Real homogeneous group \( (N,\alpha,\rho) = (N, \rho_\alpha) \)} means a nilpotent group \( N\), equipped with a derivation \( \alpha \) on its Lie algebra, and a left-invariant distance function \( \rho_\alpha \) homogeneous under \( \e^{t \alpha} \). 

%\begin{maar}
%	The triple \( (N,\alpha,\rho) \) is a \emph{homogeneous group} if 
%	\begin{lista}
%		\item \( N \), \emph{the underlying group}, is a simply connected nilpotent Lie group;
%		\item \( \rho \) is a left-invariant distance function that induces the topology of \( N\);
%		\item  \( \alpha \in \pp{der}(\mathfrak{n}) \) has the property that if the eigenvalues of \( \alpha \) are given by  \( x_1 + \I y_1,\ldots,x_k + \I y_n  \) with \( x_1 \le \cdots \le x_n \), then \( x_1=1 \);
%		\item 
%		%the maps \( \e^{\lambda \alpha} \) are metric  dilations for \(\ rho \) of factor \( \e^{s\lambda} \) for all \( \lambda \)
%	\end{lista}	
%	A homogeneous group is \emph{purely real} if the eigenvalues of \( \alpha \) are real. Homogeneous group is \emph{of Carnot type} if \( \alpha \) is the usual diagonal derivation associated to a stratification of \( N \).
%\end{maar}
\begin{maar} \label{maar:homogeneous_group}
	A pair \( (N,\alpha) \) is a \emph{homogeneous group} if \( N \) is a simply connected nilpotent Lie group and \( \alpha \) is a derivation of \( N \) so that for each eigenvalue \( \lambda \)  of \( \alpha \) it holds \( \pp{Re}(\lambda) > 0 \).
	%and it has the property that when the eigenvalues of \( \alpha \) are given by  \( x_1 + \I y_1,\ldots,x_k + \I y_n  \) with \( x_1 \le \cdots \le x_n \), then \( x_1>0 \). 
	Further, we say that a homogeneous group \( (N,\alpha) \) is
	\begin{itemize}
%		\item \emph{normalised}, if 		\( \min\{ \pp{Re}(\lambda) : \lambda \text{ eigenvalue of } \alpha \} = 1 \).
		%\( x_1=1 \),
		\item  \emph{purely real}, if the eigenvalues of \( \alpha \) are real numbers,
		\item \emph{of Carnot type} if it is purely real, if \( \alpha \) is diagonalisable over \( \mathbb{R} \), and if the eigenspace corresponding to the smallest of the eigenvalues Lie-generates \( \mathfrak{n} \).
	\end{itemize}
Two homogeneous groups \( (N_1,\alpha) \) and \( (N_2,\beta) \) are \emph{isomorphic (as homogeneous groups)} if there is an isomorphism of Lie groups \( F \colon N_1 \to N_2 \) so that \( \beta \circ F_* = F_* \circ \alpha \), where \( F_* \) is the Lie algebra isomorphism induced by \(F\).
\end{maar}

The data defining homogeneous groups exactly coincide with the data defining Heintze groups. The terms \emph{purely real Heintze group} and \emph{Heintze group of Carnot type} appear in the literature and correspond to the terms above, see for example \cite{avain:CPS} and \cite{Cornulier:qihlc}. 

For purposes of classifications up to isometry or quasi-isometry, only the purely real homogeneous groups play a role due to the result of \cite{alekseevski} presented here as Proposition \ref{prop:alekseevski} in the introduction. Hence we will always assume that the derivation has real eigenvalues even in the cases when it would be not strictly necessary.

Next we discuss homogeneous distances on homogeneous groups.

\begin{maar}
	Let \( (N,\alpha) \) be a homogeneous group.
	A distance function \( \rho \) on the set \(N\) is said to be \emph{homogeneous (for \( (N,\alpha) \))}, if \( \rho \) is left-invariant, induces the manifold topology of \( N \), and 
	%there is some \( s >0 \) so that 
	for all \( \lambda > 0 \) we have
	%\( \rho(\delta_\lambda x, \delta_\lambda y) = \lambda^s \rho(x,y) \)
	\( \rho(\delta_\lambda x, \delta_\lambda y) = \lambda \rho(x,y) \) for all \( x,y \in N \), where \( \delta_\lambda \) is the automorphism of \( N \) with the differential \( (\delta_\lambda)_* = \e^{\log(\lambda) \alpha}  \). 
	%If one may choose \( s=1 \), then \( \rho \) is said to be \emph{1-homogeneous for \( (N,\alpha) \)}.
	The triple \( (N,\alpha,\rho) \) is called \emph{a homogeneous metric group} if the distance function \( \rho \) is homogeneous for \( (N,\alpha) \).
\end{maar}

\begin{huom} \label{huom:existence_of_distance}
	In \cite[Theorem B]{seba-enrico-dilations} it is characterised when a purely real homogeneous group \( (N,\alpha) \) admits a distance \( \rho \) making it a homogeneous metric group: denoting by \( \nu \) the smallest eigenvalue of $\alpha$, a distance exists if and only if \( \nu \ge 1 \) and the restriction of \( \alpha \) to its generalised eigenspace of eigenvalue 1 is diagonalisable over \( \mathbb{R} \). Consequently, if \( (N,\alpha) \) is a homogeneous group, then for every \( \lambda > 1/\nu \), the homogeneous group \( (N,\lambda \alpha) \) admits a distance \( \rho \) making it a homogeneous metric group, and this may or may not be true for \( \lambda =1/\nu \).
\end{huom}

%%%
\begin{huom}\label{2017-09-12_kysymys}
	Given a homogeneous group \( (N,\alpha) \), all the distance functions that are homogeneous for \( (N,\alpha) \) are biLipschitz equivalent via the identity map.
	More generally, it is straightforward to prove the following statement.
	Let \( \rho \) and \( \rho' \) be two 
	distances metrising the same topological space \( M \).
	Suppose there is a transitive group of homeomorphisms acting by isometries for both of the distances.
	Suppose there is \( o \in M \) and
	a bijection \( \delta \colon M \to M \), fixing the point \(o\), and \( \lambda \in \vali{0,1} \) with 
	%%%
	\begin{equation*}  %\label{eq:}
	\rho(\delta(x),\delta(y)) = \lambda \rho(x,y) 
	\qquad \text{and} \qquad
	\rho'(\delta(x),\delta(y)) = \lambda \rho'(x,y) \pilkku \qquad \forall x,y \in M
	\piste
	\end{equation*}
	%%%
	Then \( \rho \) and \( \rho' \) are biLipschitz equivalent via the identity map of \( M\).
\end{huom}
%%%

	Due to Remark \ref{2017-09-12_kysymys}, when considering biLipschitz maps between two homogeneous groups, it is not necessary to specify the homogeneous distance functions, provided they exist, for which we refer to Remark \ref{huom:existence_of_distance}. Whenever we assume that two homogeneous groups are biLipschitz equivalent
	we mean that on both of them some homogeneous distances exist for which the metric spaces are biLipschitz equivalent.

The following result summarises the known correspondence between the quasi-iso\-met\-ries of Heintze groups and the biLipschitz maps on their boundaries. For a good exposition and list of references, see \cite[p.6]{avain:CPS}.

\begin{prop} \label{go-to-the-boundary}
	Let \( (N_1, \alpha) \) and \( (N_2, \beta) \) be homogeneous groups. Then the Heintze groups \( N_1 \rtimes_\alpha \mathbb{R} \) and \( N_2 \rtimes_\beta \mathbb{R} \) are quasi-isometric if and only if there exists \( \lambda_1,\lambda_2 > 0 \) so that \( (N_1, \lambda_1 \alpha) \) and \( (N_2, \lambda_2 \beta) \) are biLipschitz equivalent.
\end{prop}

\begin{proof}
	The two Heintze groups \( N_1 \rtimes_\alpha \mathbb{R} \) and \( N_2 \rtimes_\beta \mathbb{R} \) are quasi-isometric if and only if there are \( \lambda_1,\lambda_2 > 0 \) so that \( (N_1, \lambda_1 \alpha) \) and \( (N_2, \lambda_2 \beta) \) are quasisymmetric \cite{paulin_bord_hyp, Bonk-Schramm,Cornulier:qihlc}. 
	The constants are needed to ensure the existence of homogeneous distances, rather than quasidistances, see Remark \ref{huom:existence_of_distance}.

	If
	%for some \( \lambda_1,\lambda_2 > 0 \) we have that 
	\( (N_1, \lambda_1 \alpha) \) and \( (N_2, \lambda_2 \beta) \) are biLipschitz equivalent, then they are quasisymmetric.
	Vice versa, suppose that 
	%for some \( \lambda_1,\lambda_2 > 0 \) we have that 
	\( (N_1, \lambda_1 \alpha) \) and \( (N_2, \lambda_2 \beta) \) are quasisymmetric. 
	Without changing their biLipschitz class, we may assume that \( \alpha \) and \( \beta \) have only real eigen\-values, see \cite[Theorem C]{seba-enrico-dilations}.
	Up to changing the constants, 
	we may assume that the smallest of the eigenvalues of \( \lambda_1 \alpha \) and \( \lambda_2 \beta \) agree.
	
	If \( (N_1,\lambda_1 \alpha) \) is of Carnot type, then 
	by \cite{pansu, avain:carrasco-orlicz,avain:LD-Xie_fibers}
	we have that \( (N_1,\lambda_1 \alpha) \) and \( (N_2,\lambda_2 \beta) \) are isomorphic as homogeneous groups and thus biLipschitz equivalent.
	If \( (N_1,\lambda_1 \alpha) \) is not of Carnot type, then the quasisymmetry from \( (N_1,\lambda_1 \alpha) \)  to \( (N_2,\lambda_2 \beta) \) is a biLipschitz map,
	 by \cite{avain:carrasco-orlicz,avain:LD-Xie_fibers}.
\end{proof}

Next, we translate to our language Lemma 5.1 of \cite{avain:CPS}. 
%Remark that the statement indeed is independent of the distance \( \rho \) as indicated by the discussion above (the Hausdorff dimension is invariant under biLipschitz maps).

\begin{prop}[\cite{avain:CPS}]
	\label{prop:CPS-H-dim-lemma}
	Let \( (N,\alpha,\rho) \) be a purely real homogeneous metric group, and let \( \lambda_1 \) be the smallest of the eigenvalues of \( \alpha \).
	%Hausdorff dimension of any curve of a purely real homogeneous group \( (N,\alpha) \) is at least the smallest of the eigenvalues of \( \alpha \) (\sini{Lemma 5.1 of \cite{avain:CPS}}).
	Then the Hausdorff dimension of any non-constant curve on \( N \) is at least 
	\( \lambda_1 \), and
	%the smallest of the eigenvalues of \( \alpha \). 
	the curve \( t \mapsto \exp(t X) \) has Hausdorff dimension \( \lambda_1 \)  if \( X \) is an eigenvector of \( \alpha \) with eigenvalue \( \lambda_1 \).
\end{prop}

The next result is also a consequence of the work of \cite{avain:CPS}.
It tells us that whenever one is able to prove that some subgroups are preserved in the sense that all their left cosets are preserved, then the normalisers of these subgroups provide new invariants.

\begin{prop}[\cite{avain:CPS}] \label{lemma:CPS-corollary-about-normalisers}
	Let \( F \colon (N_1,\alpha) \to (N_2,\beta) \) be a biLipschitz map between homogeneous groups, and suppose \( A_1 \) and \( A_2 \) are subgroups of \( N_1 \) and \( N_2 \), respectively. Let \( \mathcal{N}(A_i) \) be the normaliser of \( A_i \), for \( i \in \{1,2\} \).   If for all \(p \in N_1\) we have \( F(pA_1) = F(p)A_2 \), then it holds \( F(p \mathcal{N}(A_1)) = F(p)(\mathcal{N}(A_2)) \) for all \(p \in N_1\).
\end{prop}

\begin{proof}
	%For all \( g \in N_1 \) it holds \(  \)
	Fix \( p,q \in N_1  \). Then the following are equivalent statements
	\begin{lista}
		\item \( q \in p \mathcal{N}(A_1) \).
		\item Hausdorff distance of \( qA_1 \) and \( pA_1 \) is finite.
		\item Hausdorff distance of \( F(qA_1) = F(q) A_2 \) and \( F(pA_1) = F(p) A_2 \) is finite.
		\item \( F(q) \in F(p) \mathcal{N}(A_2) \).
	\end{lista}
	Indeed, the equivalences (i)\(\LRA\)(ii) and (iii)\(\LRA\)(iv) are given by \cite[Lemma 3.2]{avain:CPS}. The equivalence (ii)\(\LRA\)(iii) is a consequence of \( F \) being a biLipschitz map.
\end{proof}

We will need to understand quotients of homogeneous groups: for that the important lemma is the following straightforward consequence of the ideas of \cite{LeDonne-Rigot2} (see their results 2.8 and 2.10 in particular).
%and below we will demonstrate and they inherit also a structure of homogeneous group with natural properties.

\begin{lemma} \label{huom:preserved-subalgebra-good-quotient}
Suppose \( H \) is a normal subgroup of a homogeneous group \( (N,\alpha) \). If \( \mathfrak{h} \) is preserved under \( \alpha \), then the quotient \( N/H \) is a homogeneous group when equipped with the induced derivation
\( \hat{\alpha} \).
% given by \( \hat{\alpha}(n H) = \alpha(n) H \) for all \( n \in N\). 
%Moreover, the quotient projection \( \pi \colon N \to N/H \) is a Lipschitz map with respect to any distances homogeneous for \( (N,\alpha) \) and \( (N/H,\hat\alpha) \), respectively.
Moreover, if \( \rho \) is a homogeneous distance on \( (N,\alpha) \), then \( \hat{\rho} \) given by
\begin{equation*}
\hat{\rho}(n H, m H) = \inf\{ \rho(n,mh) : h \in H \}
\end{equation*}
is a homogeneous distance on \( (N/H,\hat\alpha) \) for which the projection \( N \to N/H \) is a 1-Lipschitz map.
\end{lemma}

\subsection{Isometries of not necessarily solvable groups}

Above we said that two connected Lie groups \( G\) and \( H\) \emph{can be made isometric} if there are left-invariant Riemannian metrics \( \rho_G \) and \( \rho_H \)
so that \( (G,\rho_G) \) is isometric  to \( (H,\rho_H) \). 
By \cite[Proposition 2.4]{avain:KL17} 
requiring the distances \( \rho_G \) and \( \rho_H \) to be Riemannian is not restrictive: We could suppose only that the distances are left-invariant and induce the respective manifold topologies.
In any case, while quasi-isometries give a transitive relation between Lie groups, the relation by isometries is not transitive; we next wish to show an instructive example.

\begin{prop} \label{prop:CKLNO-selvennys}
	Both the groups \( \pp{SL}(2,\mathbb{R}) \) and \( \pp{PSL}(2,\mathbb{R}) \) can be made isometric to the group \( \mathbb{S}^1 \times \pp{Aff}(\mathbb{R})^+  \), but the groups \( \pp{SL}(2,\mathbb{R}) \) and \( \pp{PSL}(2,\mathbb{R}) \) cannot be made isometric (to each other).
\end{prop}

The argument  for the fact that the groups \( \pp{SL}(2,\mathbb{R}) \) and \( \pp{PSL}(2,\mathbb{R}) \) cannot be made isometric is readily recorded in \cite[Proposition 2.11]{avain:fasslerledonne}, but it goes back to Cornulier, and eventually to \cite[Theorem 2.2]{gordon-transitive}. The first part of Proposition \ref{prop:CKLNO-selvennys} may be deduced from  \cite[Theorem 3.24]{avain:CKLNO}, but in this particular example the argument of \cite{avain:CKLNO} simplifies so much that we feel it is worth giving the following elementary proof.

\begin{proof}[Proof of the first part of Proposition \ref{prop:CKLNO-selvennys}]
Let \( G \) denote either \( \pp{SL}(2,\mathbb{R}) \) or \( \pp{PSL}(2,\mathbb{R}) \), and let \( d \) be a left-invariant admissible distance on \( G \). In either case, \( G \) has the Iwasawa decomposition \( G = A N K \), where the factor \( AN \) forms a subgroup isomorphic to \( \pp{Aff}(\mathbb{R})^+ \) and \( K \) is instead isomorphic to \( \mathbb{S}^1 \). Our aim is to construct from \( d \) a new metric \( d' \) in \( G\) and find a metric \( d'' \) on \( AN \times K \) so that \( (G,d') \) is isometric to \( (AN \times K, d'') \).
	
	We define, taking the advantage of the compactness of \( K\), a new distance function on \( G \) by the formula
	\begin{equation*}
	d'(g,h) = \sup_{k \in K} d(gk, hk) \piste
	\end{equation*}
	It is trivial that \( d' \) satisfies the axioms of a distance function and that it is left-invariant.
	One may also see by a straightforward argument that  any open \( d' \)-ball contains an open \( d \)-ball. Consequently, as  \( d'(g,h) \ge d(g,h) \) for all \( g,h \in G \), then the distance  \( d' \) induces the same topology as \( d \).

	We define \( d'' \) to be the pull-back distance on \( AN \times K\) via the homeomorphism \( \omega \colon (AN \times K) \to ANK \) given by \( \omega(s,k) = s k^{-1} \). The resulting distance is left-invariant since for any fixed \( k,k_1,k_2 \in K \) and \( s,s_1,s_2 \in AN \) we have
	\begin{align*}
	d''((s,k)(s_1,k_1),(s,k)(s_2,k_2)) 
	&=
	d'(ss_1(kk_1)^{-1},ss_2,(kk_2)^{-1})
	\\
	&= 
	\sup_{k' \in K} d(ss_1k_1^{-1} k^{-1}k',ss_2,k_2^{-1}k^{-1}k')
	\\
	&= 
	\sup_{k'' \in K} d(ss_1k_1^{-1} k'',ss_2,k_2^{-1}k'')
	\\
	&= 
	\sup_{k'' \in K} d(s_1k_1^{-1} k'',s_2,k_2^{-1}k'')
	\\
	&= 
	d'(s_1k_1^{-1},s_2,k_2^{-1}) = d''((s_1,k_1),(s_2,k_2)) \piste
	\end{align*}
	In conclusion, the map \( \omega \) is an isometry between \( (G,d') \) and \( (AN \times K, d'') \).
\end{proof}

The proof of the following fact is just slightly more involved, and the details are recorded in \cite[Theorem 3.24]{avain:CKLNO}. The main difference is that one does not have a compact factor \( K\) in the Iwasawa decomposition, but instead there is a non-compact central group involved.
\begin{prop} \label{prop:CKLNO-selvennys2}
	The universal cover of the group \( \pp{SL}(2,\mathbb{R}) \) can be made isometric to the group \( \mathbb{R} \times \pp{Aff}(\mathbb{R})^+  \).
\end{prop}

Even if the transitivity of isometry-relation is  shown by Proposition \ref{prop:CKLNO-selvennys} to be false in general, we are not aware of counterexamples in the class of simply connected Lie groups. Moreover, Fact \ref{thm:graaf-CKLNO} implies the transitivity among simply connected solvable Lie groups. 
Despite Fact \ref{thm:graaf-CKLNO} some questions remain unanswered, like the following.

\begin{Kysy}
	Is there a non-solvable simply connected group \( G \) and two solvable groups \( S_1,S_2 \), so that both \( S_1 \) and \( S_2 \)  can be made isometric to \( G \) (with different metrics) and \( S_1 \) and \( S_2 \) have different real-shadow, i.e., they cannot be made isometric?
\end{Kysy}

\subsection{Algebraic tools for isometries}

\label{sec:preliminaries_isometries}

The aim of this section is to recall the results related to the real-shadow of a simply connected solvable Lie group, so that after proving Theorem \ref{thm:alkup_yleinen} in Section \ref{sec:semidirect}, we are able to link it to Fact \ref{thm:graaf-CKLNO} and real-shadows. We will make the link explicit in Corollary \ref{prop:R-shadow-and-real-part-reduction}.

We start by recalling a decomposition result which is necessary both for the construction of the real-shadow and also for the statement of Theorem \ref{thm:alkup_yleinen}.
The ingredients of its proof are recorded in \cite[Section 2]{seba-enrico-dilations} while it might be considered well known.

\begin{prop} \label{prop:decomposition_of_a_derivation}
	Let \( \alpha \) be a derivation on a Lie algebra \( \mathfrak{g}\). Then there are derivations \( \ARe \), \( \AIm \) and \( \Anil \) on \( \mathfrak{g} \) satisfying the following properties:
	\begin{lista}
		\item The maps \( \alpha \), \( \ARe \), \( \AIm \) and \( \Anil \) all pairwise commute.
		\item \( \alpha = \ARe + \AIm + \Anil \).
		%		\item The map \( \Anil \) is nilpotent.
		\item The map
		% \( \ARe + \AIm \) and 
		\( \Anil \) 
		is the nilpotent part of \( \alpha \).
		%are the semisimple and the nilpotent part of \( \alpha \), respectively.
		\item The maps  \( \ARe \) and  \( \AIm \) are semisimple.
		\item The spectrum of \( \ARe \) is real, and the spectrum of \( \AIm \) is purely imaginary.
	\end{lista} 
\end{prop}
If \( \alpha = \ad_X \) for a vector \( X \) of a Lie algebra, we denote \( \ad_\pp{s}(X) = \ARe + \AIm \) and \( \ad_\pp{si}(X) = \AIm \); In the latter, \laina{si} stands for semisimple and imaginary.

We recall some standard terminology: A Lie algebra \( \mathfrak{g} \) is said to be \emph{of type (R)} if all the eigenvalues of \( \ad_X \) are purely imaginary for all \( X \in \mathfrak{g} \). Instead, a Lie algebra is said to be \emph{completely solvable}, if it is solvable and all these eigenvalues are real.  
The Lie algebra of a simply connected Lie group \( G \) is of type (R) if and only if \( G \) has polynomial growth \cite[Theorem 1.4]{MR0349895}, i.e., the Haar measure of the powers of neighbourhoods of identity grows bounded by a polynomial.

We recall here, using a slightly different viewpoint, the method of \cite{avain:CKLNO} to determine the real-shadow of a simply connected solvable Lie group. The arguments may be found inside the proof of Theorem 4.16 in \cite{avain:CKLNO}.

\begin{lemma}
	\label{eq:choosing-a-correctly}
	Let \( \mathfrak{g} \) be a solvable Lie algebra with nilradical \( \mathfrak{n} \). Then there is a vector subspace \( \mathfrak{a} \subseteq \mathfrak{g} \) so that 
	\begin{lista}
		\item \( \mathfrak{n} \oplus \mathfrak{a} = \mathfrak{g} \),
		\item \( \ad_\pp{s}(X)(Y) =0 \) for all \( X,Y \in \mathfrak{a} \), and
		\item \( [\ad_\pp{s}(X),\ad_\pp{s}(Y)] =0 \) for all \( X,Y \in \mathfrak{a} \).
	\end{lista}
\end{lemma}

Such a subspace \( \mathfrak{a} \) is found by noticing that 
there is a Cartan subalgebra \( \mathfrak{c} \) of \( \mathfrak{g} \) so that \( \mathfrak{g} = \mathfrak{c} + \mathfrak{n} \); then \( \mathfrak{a} \) may be chosen inside \( \mathfrak{c} \) to complement \( \mathfrak{n} \).

The following statement gives naturally a very constructive definition of the real-shadow in the level of Lie algebras.

\begin{prop} \label{prop:construction_of_the_real_shadow}
	Let \( \mathfrak{g} \) be a solvable Lie algebra. Choose a vector subspace  \( \mathfrak{a} \subseteq \mathfrak{g} \) with the properties of Lemma \ref{eq:choosing-a-correctly} and let \( \pi_\mathfrak{a} \) denote the projection to \( \mathfrak{a} \) along \( \mathfrak{n} \).
	Define a map 
	\begin{equation*}
	\varphi_{\mathfrak{a}} \colon \mathfrak{g} \to \pp{der}(\mathfrak{g}) \qquad
	\varphi_{\mathfrak{a}}(X) = - \ad_\pp{si}(\pi_\mathfrak{a}(X))
	\end{equation*}
	Then
	\begin{lista}
		\item \( \varphi_{\mathfrak{a}} \) is a homomorphism of Lie algebras, with Abelian image,
		\item the graph of \( \varphi_{\mathfrak{a}} \), \( \pp{Gr}(\varphi_{\mathfrak{a}}) = \{ (X,\varphi_\mathfrak{a}(X)) \mid X \in \mathfrak{g} \} \), is a completely solvable subalgebra of \( \mathfrak{g} \rtimes \pp{der}(\mathfrak{g}) \),
		%		\item \( \pp{Gr}(\varphi) \) is completely solvable,
		\item if the vector space \( \mathfrak{g} \) is equipped with the operation defined by
		\begin{equation*}
		[X,Y]_\realshadow = [X,Y] + \varphi_{\mathfrak{a}}(X)(Y) - \varphi_{\mathfrak{a}}(Y)(X)
		\end{equation*}
		then the map \( X \mapsto (X,\varphi_{\mathfrak{a}}(X)) \) is a Lie algebra isomorphism from \( (\mathfrak{g}, [\cdot,\cdot]_\realshadow) \) to \( \pp{Gr}(\varphi_{\mathfrak{a}}) \).
	\end{lista}
	Moreover, for every vector subspace \( \mathfrak{a}' \subset \mathfrak{g} \) as in Lemma \ref{eq:choosing-a-correctly} we have that \( \pp{Gr}(\varphi_{\mathfrak{a}}) \) is isomorphic to \( \pp{Gr}(\varphi_{\mathfrak{a}'}) \).
\end{prop}

\begin{maar}
	Let \( \mathfrak{g} \) be a solvable Lie algebra. Its \emph{real-shadow} is the Lie algebra \( \pp{Gr}(\varphi_\mathfrak{a}) \) constructed as in Proposition \ref{prop:construction_of_the_real_shadow}.
\end{maar}

The main result of \cite{avain:CKLNO} regarding this construction is that Fact \ref{thm:graaf-CKLNO} indeed holds for such a construction.

\begin{huom} \label{rmk:nilshadow_is_real}
	In many applications of low dimension, there is an Abelian subalgebra \( \mathfrak{a} \) complementary to the nilradical \( \mathfrak{n} \) of \( \mathfrak{g} \). Then such \( \mathfrak{a} \) trivially satisfies Lemma \ref{eq:choosing-a-correctly} and can be used to construct the real-shadow. Another remark is that if \( \mathfrak{g} \) is of type (R), then \( \ad_\pp{si}(X) = \ad_\pp{s}(X) \) for any \( X \in \mathfrak{g} \), and consequently, the real-shadow of a Lie algebra of type (R) is its nilshadow as defined in \cite{avain:DER03}.
\end{huom}

\subsection{Algebraic tools for quasi-isometries}

\label{sec:algebraic_tools_for_QI}

When considering the class of simply connected solvable Lie groups, the algebraic tools relevant for our study of groups of dimension 4 and 5 up to quasi-isometry are the following invariants:
\begin{enumerate}[label=(Inv-\arabic*)]
	\item \label{inv:A} Carnot groups are quasi-isometrically distinct among themselves by Pansu's Theorem \cite{pansu}. More generally, \cite{pansu} implies that if two simply connected nilpotent Lie groups are quasi-isometric, their associated Carnot groups are isomorphic.
	\item \label{inv:B} For nilpotent groups, the Betti numbers (by \cite{Shalom}) and more generally the Lie algebra cohomology rings (by \cite{Sauer2006}) are quasi-isometry invariants.
	\item \label{inv:C} For the groups of polynomial growth, their degree of growth is quasi-isometry invariant. It is because this degree is the Hausdorff dimension of the asymptotic cone.
	\item \label{inv:D} Two simply connected solvable groups are quasi-isometric if and only if their real-shadows are quasi-isometric. This is because these groups are quasi-isometric to their real-shadows by Fact \ref{thm:graaf-CKLNO}.
	\item \label{inv:E} Topological dimension of the asymptotic cone, called \emph{cone dimension}, is a quasi-isometry invariant.
\end{enumerate}
As nilshadows were already treated above, we turn the attention here to the cone dimension. Cornulier proved in \cite{MR2399134} that the cone dimension of a simply connected solvable Lie group agrees with the dimension of the exponential radical of the group. 
We turn this result into the following observation.

\begin{prop} \label{prop:cone-dimensions-algebraically}
	Let \( G \) be a simply connected completely solvable Lie group with Lie algebra \( \mathfrak{g} \). Then the cone dimension of \( G \) equals the codimension of the subspace \( \bigcap_{n \ge 1} \mathfrak{g}^{n} \) of \( \mathfrak{g} \), where the subspace \( \mathfrak{g}^n \) denotes the \(n \)th term in the lower central series of \nolinebreak \( \mathfrak{g} \).
\end{prop}

\begin{proof}
	By \cite{avain:osin-exp} (see also \cite[Theorem 6.1]{MR2399134}), the exponential radical  \(R\) of \( G \) is a closed connected normal subgroup of \( G \), the quotient group \( G/R \) has polynomial growth, and there is no closed connected normal subgroup \( R' \) so that the quotient \( G/R' \) would be of polynomial growth and have strictly larger dimension.
	
	By \cite[Theorem 1.1]{MR2399134}, the cone dimension of \( G \), denoted by \( \conedim(G) \), equals to the codimension of the exponential radical of \( G\). By the above
	\begin{equation*}
	\conedim(G) = \max \{\dim (\mathfrak{g}/\mathfrak{r} ) : \mathfrak{r} \text{ ideal of } \mathfrak{g} \text{ and } \mathfrak{g}/\mathfrak{r} \text{ is of type (R)} \}
	\end{equation*}
	Moreover, since \( G \) is completely solvable, a quotient of \( \mathfrak{g} \) is of type (R) if and only if it is nilpotent.
	
	The terms of the lower central series are nested vector subspaces of \( \mathfrak{g} \), and the condition \( \mathfrak{g}^n=\mathfrak{g}^{n+1} \) for some \(n\) implies that \( \mathfrak{g}^n = \mathfrak{g}^k \) for all \( k \ge n \). Thus there is \( N \in \mathbb{N} \) so that \( \bigcap_{n \ge 1} \mathfrak{g}^{n} = \mathfrak{g}^N \). The quotient \( \mathfrak{g}/\mathfrak{g}^N \) is nilpotent, and we will show its dimension is maximal.
	Let \( \mathfrak{q} \) be an ideal of \( \mathfrak{g} \) so that \( \mathfrak{g}/\mathfrak{q} \) is nilpotent of step \( s \). It is enough to show \( \mathfrak{g}^N \subset \mathfrak{q} \). Assuming the contrary, we have a non-zero vector \( X \in \mathfrak{g}^N \pois \mathfrak{q} \). Because \( \mathfrak{g}^N = \mathfrak{g}^k \) for all \( k \ge N \), we may express \( X \) as a bracket of arbitrary length. More precisely \( X \)  may be expressed as a linear combination of a terms of the form
	\( \ad_{X_1} \circ \cdots \circ  \ad_{X_s}(X_{s+1}) \) for some \( X_i \in \mathfrak{g} \). It holds \( X_i \not\in \mathfrak{q} \) since \( \mathfrak{q} \) is an ideal. Hence when \( X \) is considered as a non-zero element of the quotient \( \mathfrak{g} / \mathfrak{q} \), it can be expressed as a bracket of length \(s+1\), which contradicts the nilpotency step of the quotient.
\end{proof} 
The above result implies that the cone dimension can be algorithmically calculated at the Lie algebra level.

%%%%%%%%%%%%%%%%%%%%%%%%%%%%%%%%%%%%%%%%%
\section{On biLipschitz maps of homogeneous groups}

\label{sec:lemma_of_Seba}

It is conjectured that if two purely real Heintze groups are quasi-isometric, then they are isomorphic.
Many quasi-isometry invariants are known, 
but still there are non-isomorphic pairs of  purely real Heintze groups that are not distinguished by those invariants.
In this section we present new quasi-isometry invariants for purely real Heintze groups: we prove Theorem \ref{thm:computable-reachability-sets2}. 
Our analysis is based on the fact that two purely real Heintze groups are quasi-isometric if and only if their parabolic boundaries are biLipschitz equivalent, see Proposition \ref{go-to-the-boundary}.

In Section \ref{sec:abelian-theorem-new} we prove Theorem~\ref{thm:computable-reachability-sets2}.(i) stating that biLipschitz equivalent homogenous groups are quasi-isometric when equipped with Riemannian distances.
One important consequence, see also Theorem~\ref{lause:abelianity_preserves_with_Seba} in appendix, is that the family of purely real Heintze groups with Abelian nilradical is closed under quasi-isometries among the family of purely real Heintze groups. In addition, the quasi-isometry relations within the family of purely real Heintze groups with Abelian nilradical are completely understood by the results of Xie \cite{MR3180486}.

In Section \ref{sec:reachability}, we prove that on the level of the boundary, the set of points reachable by curves of a given Hausdorff dimension can be algebraically computed and hence used as an invariant. This leads to Theorem \ref{thm:computable-reachability-sets2}.(ii). Such a result will enable us to distinguish up to quasi-isometry some examples of low dimension that we discuss in Section \ref{sec:examples_heintze}.

We recall that in this paper, we will always use the convention that if \( N,H,G, \ldots \) are Lie groups, then the fraktur letters \( \mathfrak{n}, \mathfrak{h}, \mathfrak{g}, \ldots \) denote their Lie algebras, and vice versa.

\subsection{Homogeneous biLipschitz implies Riemannian quasi-isometric}
\label{sec:abelian-theorem-new}

In this section we prove Theorem~\ref{thm:computable-reachability-sets2}.(i). We follow a suggestion of Pansu for treating arbitrary homogeneous groups.
% for informing us of this argument. 
However, in Section~\ref{sec:abelian-theorem}, we consider the case where one of the groups is Abelian. 
It is a less general setting, but the proof is direct and might be of independent interest.

\begin{proof}[Proof of Theorem~\ref{thm:computable-reachability-sets2}.(i)]
		Given a metric space \( (M,d) \) and \( \ell > 0 \), we recall from \cite{avain:CKLNO} the definition of derived semi-intrinsic metric with parameter \( \ell \) as
	\begin{equation*}
	d_{[\ell]}(p,q) = \inf \Bigl\{ \sum_{j=1}^k d(x_j,x_{j-1}) \mid x_0,\ldots,x_k \in M,\: x_0=p,\: x_k=q,\: d(x_j,x_{j-1}) \le \ell \Bigr\} \piste
	\end{equation*}
	It follows immediately from the definition, that if a map \( f \colon (M,d) \to (M',d') \) is an \(L\)-Lipschitz-map with \( L \ge 1 \), then \( d'_{[\ell]}(f(p),f(q)) \le L d_{[\ell/L]}(p,q) \).
	By \cite[Lemma 2.3]{avain:CKLNO}, for a homogeneous metric group \( (N,d) \) and \( \ell > 0 \), the function \( d_{[\ell]} \) is a proper quasi-geodesic distance function inducing the topology of \( N \). Thus, \( d_{[\ell]} \) is quasi-isometric to any left-invariant Riemannian distance on \(N\).
	
	We conclude that if \( F \colon N_1 \to N_2 \) is an \(L\)-biLipschitz map between homogeneous metric groups \( (N_1,\alpha,d^\alpha) \) and \( (N_2,\beta,d^\beta) \), then
	%Let \( D_1 \) and \( D_2 \) be any left-invariant Riemannian distances on the Lie groups \( N_1 \) and \( N_2 \), respectively. We shall prove that \( F \colon (N_1,D_1) \to (N_2,D_2) \) is a quasi-isometry.
%	Let \( \ell > 0 \) be such that both the closed balls \( \bar{B}_{\rho_\alpha}(1_{N_1},L\ell) \) and \( \bar{B}_{\rho_\beta}(1_{N_2},L\ell) \) are compact.
%	the above observation implies the following inequalities 
	 the derived semi-intrinsic metrics satisfy the following inequalities:
	\begin{equation*}
	\frac{1}{L} d_{[L \ell]}^\alpha (x,y) \le d_{[\ell]}^\beta (F(x),F(y)) \le L d_{[\ell/L]}^\alpha (x,y) \piste
	\end{equation*}
%	On the other hand, since \( d_{[L \ell]}^\alpha \) and \( d_{[\ell/L]}^\alpha \) both are left-invariant proper quasigeodesic metrics on the Lie group \( N_1 \), they are both quasi-isometric to \( D_1 \) via the identity map. Similarly, \( d_{[\ell]}^\beta \) is quasi-isometric to \( D_2 \) via the identity map. 
Therefore, if \( D_1\) and \( D_2\) are left-invariant Riemannian distances, then the map \linebreak \( F \colon (N_1,D_1) \to (N_2,D_2) \) is a quasi-isometry.
\end{proof}

\subsection{Reachability sets}
\label{sec:reachability}

Considering homogeneous groups up to biLipschitz equi\-va\-len\-ce, one obvious invariant is the set of those points that can be reached by curves starting from the identity element and having Hausdorff dimension at most \(s \), for some fixed \( s \ge 1 \) (notice that curves have Hausdorff dimension at least 1). When \( (N,\alpha) \) is a homogeneous group, we denote
\begin{equation*}
R(s) 
= \{ \gamma(1) \mid \gamma \in \mathcal{C}^0([0,1],N), \: \gamma(0) = 1_N, \:\Hdim(\gamma([0,1])) \le s \}
\end{equation*}
As one might expect and as we shall now prove, such a set may be computed as the subgroup \( \reach{N}{\alpha}{s} < N \)  corresponding to the subalgebra \( \pp{LieSpan}(\bigoplus_{0 < \lambda \le s } V_\lambda) \). Here \( \mathfrak{n} = \bigoplus_{\lambda >0 } V_\lambda \) is the decomposition of the Lie algebra by the generalised eigenspaces of the derivation~\( \alpha \). 
The fact that \( R^\alpha(s) = \reach{N}{\alpha}{s} \) makes this set into a practically usable invariant.

\begin{lause} \label{thm:computable-reachability-sets}
	Let \( (N,\alpha) \) be a purely real homogeneous group. Then \( R(s) = \reach{N}{\alpha}{s} \) for every \( s \ge 1 \).
\end{lause}

\begin{proof}
	Fix \( s \ge 1 \). Using the Orbit Theorem, one may show (see \cite[Proposition~2.26]{bellettini}) the following. Suppose \( W \) is subset of a Lie algebra \( \mathfrak{g} \) so that \( W \) is invariant under scalar multiplication, i.e., \( \mathbb{R} W = W \), and so that no proper subalgebra of \( \mathfrak{g} \) contains \( W \). Then \( \bigcup_{k=1}^\infty (\exp(W))^k \) has non-empty interior in \( G \), and since it is also a subgroup it holds \( \bigcup_{k=1}^\infty (\exp(W))^k = G \). Applying this observation to \( W= \bigcup_{0 < \lambda \le s } V_\lambda \) and \( G = \reach{N}{\alpha}{s} \), we get that every element of \( \reach{N}{\alpha}{s} \) is a finite product of exponentials of vectors \( X \in \bigcup_{0 < \lambda \le s } V_\lambda \). Thus, to show  that \( R(s) \supset \reach{N}{\alpha}{s} \), we only need to see that the flow lines \( t \mapsto \exp(tX) \)  have Hausdorff dimension at most \(s\), for \( X \in \bigcup_{0 < \lambda \le s } V_\lambda \). By \cite[Lemma 5.1]{avain:CPS}, we may assume that \( X \) is an eigenvector of \( \alpha \) with eigenvalue \( \lambda \le s \). Fix a homogeneous distance \( \rho \),
	and set \( L = \exp(\mathbb{R} X) \).
	Identifying  \( L \) with \( \mathbb{R} \), we get a distance to \( \mathbb{R} \) that is homogeneous under the family of dilations induced by \( \alpha \). Hence by Remark \ref{2017-09-12_kysymys},
%	By the homogeneity, 
	%\( t \mapsto \exp(tX) \) is a rectifiable curve of the snowflaked distance \( \rho^{1/\lambda} \), hence it has Hausdorff dimenion 1 and therefore it has Hausdorff dimension \( \lambda \) for the distance \( \rho \).
	\( (L,\rho) \) is biLipschitz equivalent to \( (\mathbb{R},\norm{\cdot}^{1/\lambda}) \) and hence it has Hausdorff dimension \( \lambda \).

	To prove that \( R(s) \subset \reach{N}{\alpha}{s} \), denote 
	% subalgebra \( S_i \) of \( N_i \) define recursively \( \mathcal{N}_i^{[0]}(S_i) \coloneqq S_i \) and let then \( \mathcal{N}_i^{[j]}(S_i) \) denote the normaliser of the subalgebra \( \mathcal{N}_i^{[j-1]}(S_i) \). 
	\( H_0 = \reach{N}{\alpha}{s} \) and let then recursively  \( H_k \) denote the normaliser of \( H_{k-1} \).
	Consider the finite chain of subgroups \( \reach{N}{\alpha}{s} = H_0 < H_1 < \cdots < H_m = N \), where \( m \ge 1 \) is the first integer so that the repeated normaliser is the full space. Since nilpotent Lie algebras don't have non-trivial self-normalising subalgebras, such \( m\) exists.
%	We shall prove inductively that given \( x=\gamma(1) \in R(s) \), then the corresponding curve \( \gamma \) does not leave \( H_k \) for any \( 0 \le k \le m \). 
Fix a continuous curve \( \gamma \colon [0,1] \to N \) with \( \Hdim(\gamma([0,1])) \le s \) and \( \gamma(0) = 1_N \). We shall prove inductively that \( \gamma \) does not leave \( H_k \) for any \( 0 \le k \le m \).
	
	The case \( k=m \) of the induction is trivial. So we assume \( \gamma \) does not leave \( H_k \) for some \( k \le m \). Since \( H_{k-1} \) is normal in \( H_k \), we may consider the quotient \( H_k/H_{k-1} \).
	Observe that if a derivation \( \alpha \) preserves a subalgebra \( \mathfrak{q} < \mathfrak{n} \), then \( \alpha \) necessarily preserves the normaliser of \( \mathfrak{q} \).
	Therefore, since \( \alpha \) preserves \( \reach{N}{\alpha}{s} \), then, by induction and Lemma \ref{huom:preserved-subalgebra-good-quotient}, the  quotient \( H_k/H_{k-1} \) is a homogeneous group. Moreover, the
	 curve \( \gamma \) projects to the curve \( \pi \circ \gamma \) of \( H_k/H_{k-1} \), and Lemma \ref{huom:preserved-subalgebra-good-quotient} guarantees that the Hausdorff dimension of \( \pi(\gamma([0,1])) \) is at most the Hausdorff dimension of \( \gamma([0,1]) \), so at most \( s \). 
	 
Next, remark that all the generalised eigenspaces of \( \alpha \) corresponding to eigenvalues less or equal to \( s \) are contained in the Lie algebra of \( \reach{N}{\alpha}{s} \), thus they are contained in \( H_{k-1} \). This shows that all the eigenvalues of the derivation induced to \( H_k/H_{k-1} \) are strictly larger than \( s \). Therefore,
	 by Proposition \ref{prop:CPS-H-dim-lemma}, either \( \pi \circ \gamma \) is constant or the
	  Hausdorff dimension of \( \pi(\gamma([0,1])) \) is strictly larger than \(s\). Since the second case is ruled out, the curve \( \pi \circ \gamma \)  must be constant, i.e., \( \gamma([0,1]) \subset H_{k-1} \) as the induction requires. We conclude that \( \gamma \) does not leave \( H_0 =  \reach{N}{\alpha}{s} \) and hence  \( R(s) \subset \reach{N}{\alpha}{s} \). 
	  %Proposition \ref{prop:CPS-H-dim-lemma} tells that it is enough to see that in the quotient group \( H_k/H_{k-1} \), the lowest eigenvalue of the induced derivation is strictly bigger than \( s\). 	 
%	  and the induction is complete and \( \gamma \) does not leave \( H_0 = R(d) \), and this means \( R(s) \subset \reach{N}{\alpha}{s} \).
\end{proof}

\begin{proof}[Proof of Theorem \ref{thm:computable-reachability-sets2}.(ii)]
	As the set \( R(s) \) is metrically defined, we get Theorem~\ref{thm:computable-reachability-sets2}.(ii) as immediate corollary of Theorem \ref{thm:computable-reachability-sets} when applying also Proposition \ref{lemma:CPS-corollary-about-normalisers}.
\end{proof}

\subsection{Examples}

\label{sec:examples_heintze}

In this section we present some examples of pairs of Heintze groups
trying to distinguish them up to quasi-isometry
% and apply 
using the results that we proved. 
%To our best knowledge, these pairs could not be distinguished with the tools found in the literature. In particular, \cite[Corollary 1.2]{avain:CPS} does not distinguish them.
\begin{itemize}
%	\item[Ex \ref{esim:engel-and-heisR}] This is a pair of 5-dimensional Heintze groups with different nilradical, but identical diagonal derivation (when identifying the vector spaces). Both the parts (i) and (ii) of Theorem \ref{thm:computable-reachability-sets2} distinguish them.
	\item[Ex \ref{esim:pair-noCPS}] This is a pair of 7-dimensional Heintze groups with identical nilradical and derivations with identical diagonal form. Theorem \ref{thm:computable-reachability-sets2}.(ii) distinguishes them.
		\item[Ex \ref{esim:5D-heintze-sublin}] This is a pair of 5-dimensional Heintze groups with identical nilradical. This pair cannot be distinguished even with the new invariants we presented.
	\item[Ex \ref{esim:6D-heintzes-different-nilrad}] This is a pair of 7-dimensional Heintze groups with different nilradical, but identical diagonal derivation. Theorem \ref{thm:computable-reachability-sets2}.(i) distinguishes them.
	\item[Ex \ref{esim:HeistimesHeis}] 
	This is a pair of 10-dimensional Heintze groups with identical nilradical and derivations with identical diagonal form. Here the reachability sets don't distinguish the pair directly, but the normalisers can be used to distinguish them.
	\item[Ex \ref{esim:6D-heintzes-different-nilrad2}] This is a pair of 7-dimensional Heintze groups with different nilradical, but identical diagonal derivation. This pair cannot be distinguished even with the new invariants we presented.
%	This is an example of 10-dimensional Heintze groups for which all the reachability sets agree even if the derivations are diagonal. Here the normalisers can be used in distinguishing them.
\end{itemize}
%Remarkably, Examples \ref{esim:pair-noCPS} and \ref{esim:HeistimesHeis} can be distinguished by the methods of this paper. We will show how when introducing these examples.

In the next examples, we use the notation \( \pp{Heis} \) for the standard Heisenberg group
%, \( \pp{Engel} \) for the Engel group, 
and \( \pp{Heis}(5) \) for the 5-dimensional Heisenberg group. These are indexed by \( A_{3,1} \)
%, \( A_{4,1} \) 
and \( A_{5,4} \), respectively, in \cite{avain:Patera}, see also Section \ref{sec:4D} and Section \ref{sec:5D} later.
	%Table \ref{4D-luokittelu-real} and Table \ref{table:5d-poly-growth} later.}

%\begin{Esim} \label{esim:engel-and-heisR}
%	Fix \( a > 1 \). Consider a 4-dimensional vector space \( V \) with a basis \( X_1,\ldots,X_4 \), and the linear map whose matrix in this basis is
%	\begin{align*}
%	\alpha = \matriisi{cccc}{2a+1 & 0 & 0 & 0 
%		\\ 0 & 1 & 0 & 0
%		\\ 0 & 0 & 1+a & 0
%		\\ 0 & 0 & 0 & a
%	}
%	\end{align*}
%	To the vector space \( V \) one can put the Lie algebra structure of \( \pp{Heis} \times \mathbb{R} \) by defining the bracket \( [X_4,X_2] = X_3 \). If we in addition introduce the bracket \( [X_4,X_3] = X_1 \), we have the Lie algebra of the Engel group. The linear map defined by the matrix \( \alpha \) is in both cases a derivation. The corresponding Heintze groups are given by elements of the families \( A_{5,19} \) (for \( (\pp{Heis} \times \mathbb{R}) \rtimes_\alpha \mathbb{R}\)) and \( A_{5,30} \) (for \( \pp{Engel} \rtimes_\alpha \mathbb{R}\)) in the classification \cite{avain:Patera}.
%	%, so they are in particular non-isomorphic Heintze groups. 
%	% it is tedious and seems useless to give the parameter values, like \( A_{5,19}^{a+1,2a+1} \).
%	%They cannot be distinguished by the characteristic polynomial or Theorem \ref{lause:abelianity_preserves_with_Seba}, but the 
%	Reachability sets with Hausdorff dimension \( 1+a \) may be used to distinguish this pair, since \( R(1+a) \) is everything  for the Engel group  and it is not for \( \pp{Heis} \times \mathbb{R} \).
%\end{Esim}

\begin{Esim} \label{esim:pair-noCPS}
Consider the Lie group \( N = \pp{Heis} \times \mathbb{R}^3 \) and two derivations on it
	\begin{equation*}
	\alpha = \diag(1,2,3,4,5,9) \qquad \text{and} \qquad \beta=\diag(4,5,9,1,2,3)
	\end{equation*}
	Then the Heintze groups \( N \rtimes_\alpha \mathbb{R} \) and \( N \rtimes_\beta \mathbb{R} \) are not isomorphic: if they were, then \( \alpha \) should be conjugate to \( \beta \) by an automorphism of \( \mathfrak{n} \) (see for example \cite[Proposition 4.7]{hakavuori2020gradings}). However, there is a unique linear endomorphism of \( \mathfrak{n} \) that conjugates \( \alpha \) to \( \beta\) and it is not an automorphism. 
	
	The invariant \( R(2) \) distinguishes these homogeneous groups \( (N,\alpha) \) and \( (N,\beta) \) by Theorem \ref{thm:computable-reachability-sets}, as these sets have topological dimension 3 for \( (N,\alpha) \) and  2 for \( (N,\beta) \).
\end{Esim}

\begin{Esim} \label{esim:5D-heintze-sublin}
	Consider the 4-dimensional Lie algebra \( \pp{Heis} \times \mathbb{R} \) given by a basis \( X_1,\ldots,X_4 \) with the only non-trivial bracket being \( [X_1,X_2] = X_3 \). Consider, for every parameter \( a >1 \), the two linear maps given by matrices
	\begin{equation*}
	\alpha = \matriisi{cccc}{a-1 & 0 & 0 & 0 
		\\ 0 & 1 & 0 & 0
		\\ 0 & 0 & a & 1
		\\ 0 & 0 & 0 & a
	}
	\qquad \text{and} \qquad
	\beta=
	\matriisi{cccc}{a-1 & 0 & 0 & 0 
		\\ 0 & 1 & 0 & 0
		\\ 0 & 0 & a & 0
		\\ 0 & 0 & 0 & a
	}
	\end{equation*}
	Both these maps are derivations of \( \pp{Heis} \times \mathbb{R} \) with strictly positive eigenvalues, hence they define two 5-dimensional Heintze groups \(  (\pp{Heis} \times \mathbb{R}) \rtimes_\alpha \mathbb{R} \) and \(  (\pp{Heis} \times \mathbb{R}) \rtimes_\beta \mathbb{R} \). These Heintze groups are non-isomorphic, as they are the groups \( A_{5,20}^a \) and \( A_{5,19}^{a,a} \) in the classification \cite{avain:Patera}.
	We are not aware of any method of distinguishing these Heintze groups up to quasi-isometry. We remark that while the Jordan-forms of the derivations are different, the Jordan form is not proven to be invariant in this generality.
	We also remark that these groups are sublinearly biLipschitz equivalent, by a result of Cornulier \cite[Theorem 1.2]{MR3006687}, see also \cite[Theorem 3.2]{avain:Pal20}.
	%
	%	\puna{5D Heintzes are too much to ask.} Give examples what are open.	
\end{Esim}

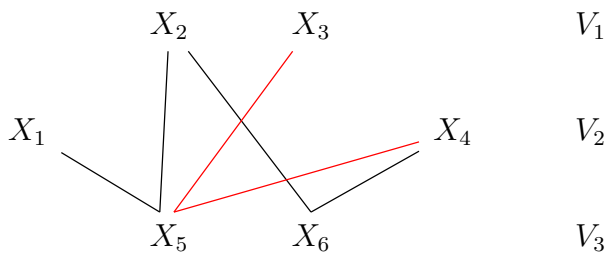
\begin{figure}[h]
	\begin{equation*}
	\begin{tikzcd}
	& X_2 \arrow[dd, dash, end anchor=110] \arrow[ddr, dash, end anchor=90] & X_3 \arrow[ddl, dash, end anchor=80, red]  & & V_1
	%\arrow[dddl, dash, end anchor=70, bend left=20, red]
	\\
	X_1 \arrow[dr, dash, end anchor=110] &   
	% \arrow[dd, dash, end anchor=70, bend left=30, red] }
	& & X_4 \arrow[dl, dash, end anchor=90] \arrow[dll, dash, end anchor=80, red] & V_2
	\\
	& X_5 & X_6 && V_3
	\end{tikzcd}
	\end{equation*}
	\caption{The graph representing schematically the bracket relations and positive gradings \(V_1 \oplus V_2 \oplus V_3\) of Example \ref{esim:6D-heintzes-different-nilrad}.}
	\label{fig:bracket-relations}
\end{figure}

\begin{Esim} \label{esim:6D-heintzes-different-nilrad}
	Consider the 6-dimensional vector space with a basis \( X_1,\ldots,X_6 \) with two different structures of Lie algebra:
	Let \( \mathfrak{n}_1 \) be the Lie algebra  given by the non-trivial bracket relations
	\begin{equation*}
	[X_1,X_2] = X_5 \qquad
	[X_2,X_4] = X_6 \qquad
	%	[X_3,X_4] = X_5 
	\end{equation*}
	This is denoted by \( L_{6,8} = L_{5,8} \times \mathbb{R} \) in the classification \cite{deGraaf-2007-dim_6_nilpotent_lie_algebras}.
	Let \( \mathfrak{n}_2 \) be instead given by
	\begin{equation*}
	[X_1,X_2] = X_5 \qquad
	[X_2,X_4] = X_6 \qquad
	[X_3,X_4] = X_5 
	\end{equation*}
	This is denoted by \( L_{6,22}(0) \) in the classification \cite{deGraaf-2007-dim_6_nilpotent_lie_algebras}. 
	
	The linear map \( \alpha = \diag(2,1,1,2,3,3) \) in this basis is a derivation for both of these Lie algebra structures. For a schematic presentation, see Figure \ref{fig:bracket-relations}.
	
 The homogeneous groups \((N_1,\alpha)\) and \( (N_2,\alpha) \) cannot be biLipschitz-distinguished by Theorem \ref{thm:computable-reachability-sets2}.(ii) but they can by Theorem \ref{thm:computable-reachability-sets2}.(i) since the Lie algebras in question are stratifiable (even though homogeneous structures given are not of Carnot type).
\end{Esim}

\begin{Esim} \label{esim:HeistimesHeis}
	Consider the 10-dimensional Lie algebra \( N = \pp{Heis}(5) \times \pp{Heis}(5) \) expressed as a vector space spanned by \( X_1,\ldots,X_{10} \) with the non-trivial bracket relations
	\begin{equation*}
	[X_1,X_2] = X_5 \qquad [X_3,X_4] = X_5 \qquad [X_6,X_7] = X_{10} \qquad [X_8,X_9] = X_{10}
	\end{equation*}
	Consider the two derivations given by matrices
	\begin{equation*}
	\alpha = \diag(1,7,3,5,8,2,6,4,4,8)
	\qquad \text{and} \qquad
	\beta = \diag(1,7,4,4,8,2,6,3,5,8)
	\end{equation*}
	The resulting Heintze groups \( G_\alpha = N \rtimes_\alpha \mathbb{R} \) and \( G_\beta = N \rtimes_\beta \mathbb{R} \) are not isomorphic: If they were, \( \alpha \) and \(\beta\) should be conjugate by an automorphism of \( N \) (again, \cite[Proposition 4.7]{hakavuori2020gradings}), but this is impossible. 
	Indeed, the conjugating automorphism is forced to 
	%permute the sets \( \{X_3,X_4\} \) and \( \{X_8,X_9\} \) 
	map \( X_3 \mapsto X_8 \) and \( X_4 \mapsto X_9 \)
	while in the same time keeping the basis-vectors \( X_1 \) and \(X_2\)
	%, \(X_6\) and \(X_7 \) 
	fixed, which is not conceivable. 
%	Indeed, the conjugating automorphism is forced to map \( X_3 \to X_8 \) and \( X_4 \to X_9 \), but \( [X_3,X_4] = X_{10} \) while \( [] \) ????
	
	Distinguishing these spaces is a bit more involved and demonstrates the combined power of Theorem \ref{thm:computable-reachability-sets} and Lemma \ref{lemma:CPS-corollary-about-normalisers}. We cannot distinguish them directly via the reachability sets of prescribed Hausdorff dimension. However, the following works.
	
	Suppose \( F \colon (N,\alpha) \to (N,\beta) \) is a biLipschitz map. 
	Then \( F \) must map \( \reach{N}{\alpha}{6} \) to \( \reach{N}{\beta}{6} \).
	%The set of points reachable by curves of Hausdorff dimension less or equal to 6 is preserved. 
	By Theorem \ref{thm:computable-reachability-sets}, \( \reach{N}{\alpha}{6} \) and \( \reach{N}{\beta}{6} \) both agree with the subgroup spanned by all the other basis vectors except \( X_2 \). This subgroup is again a homogeneous group, and it is Lie isomorphic to \( \mathbb{R} \times \pp{Heis} \times \pp{Heis}(5) \). The original biLipschitz map induces a biLipschitz map of this subgroup equipped with the two different homogeneous structures (the derivations), call these groups \( (N_0,\alpha_0) \) and \( (N_0,\beta_0) \).	
	For these two homogeneous groups, consider now the subgroups 
	\begin{equation*}
%	R^{\alpha_0}(4) 
	\reach{N_0}{\alpha_0}{4}
	 =  \braket{X_1, X_3, X_6, X_8, X_9, X_{10}} \quad \text{and} \quad
%	R^{\beta_0}(4) 
\reach{N_0}{\beta_0}{4}
	 =  \braket{X_1, X_3, X_4, X_5, X_6, X_8}
	\end{equation*}
	These are both isomorphic to the group \( \mathbb{R}^3 \times \pp{Heis} \), so we did not yet distinguish the groups. However, the normalisers of these subgroups inside \( (N_0,\alpha_0) \) and \( (N_0,\beta_0) \) are preserved by Lemma \ref{lemma:CPS-corollary-about-normalisers}. These normalisers are
	\begin{equation*}
	\mathcal{N}(\reach{N_0}{\alpha_0}{4})  =  \reach{N_0}{\alpha_0}{4} \oplus \braket{X_5, X_7} \quad \text{and} \quad
	\mathcal{N}(\reach{N_0}{\beta_0}{4})  =  \reach{N_0}{\beta_0}{4} \oplus \braket{X_{10}}
	\end{equation*}
	which have different topological dimension and this prevents the existence of a \linebreak biLipschitz map.
\end{Esim}

\begin{figure}[h]
	\begin{equation*}
	\begin{tikzcd}
	& X_1 \arrow[dd, dash, end anchor=110] \arrow[drr, dash, end anchor=140, start anchor=330]
	\arrow[ddr, dash, end anchor=90, start anchor=300]
	& X_2 \arrow[dr, dash, end anchor=140, start anchor=320] 
	\arrow[dd, dash, end anchor=120, red]
	& & V_1
	\\
	X_4 \arrow[dr, dash, end anchor=110] \arrow[drr, dash, end anchor=120, red] &  & & X_3
	\arrow[dl, dash, end anchor=90, start anchor=210]  & V_2
	\\
	& X_6 & X_5 && V_3
	\end{tikzcd}
	\end{equation*}
	\caption{The graph representing schematically the bracket relations and positive gradings \(V_1 \oplus V_2 \oplus V_3\) of Example \ref{esim:6D-heintzes-different-nilrad2}.}
	\label{fig:bracket-relations2}
\end{figure}
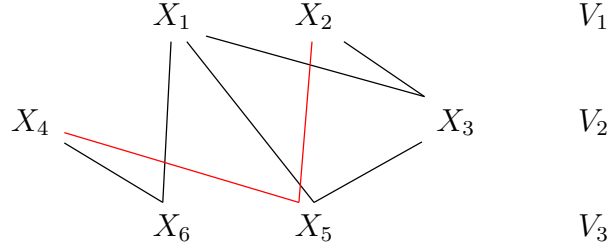

\begin{Esim} \label{esim:6D-heintzes-different-nilrad2}
	Consider the 6-dimensional vector space with a basis \( X_1,\ldots,X_6 \) with two different structures of Lie algebra:
	Let \( \mathfrak{n}_1 \) be the Lie algebra  given by the non-trivial bracket relations
	\begin{equation*}
	[X_1,X_2] = X_3 \qquad
	[X_1,X_3] = X_5 \qquad
	[X_1,X_4] = X_6 \qquad
	[X_2,X_4] = X_5 \piste
	\end{equation*}
	This is denoted by \( L_{6,23} \) in the classification \cite{deGraaf-2007-dim_6_nilpotent_lie_algebras}.
	Let \( \mathfrak{n}_2 \) be instead given by only the first three from above, i.e.,
	\begin{equation*}
	[X_1,X_2] = X_3 \qquad
	[X_1,X_3] = X_5 \qquad
	[X_1,X_4] = X_6 \piste
	\end{equation*}
	This is denoted by \( L_{6,25} \) in the classification \cite{deGraaf-2007-dim_6_nilpotent_lie_algebras}.

	The linear map \( \alpha = \diag(1,1,2,2,3,3) \) in this basis is a derivation for both of these Lie algebra structures.
	For a schematic presentation, see Figure \ref{fig:bracket-relations2}. 
	
	The homogeneous groups \((N_1,\alpha)\) and \( (N_2,\alpha) \) cannot be biLipschitz-distinguished by any method we know:
	The Lie algebra \( L_{6,25} \) is the associated Carnot algebra of \( L_{6,23} \) and the simply connected nilpotent Lie groups associated are not distinguished by the known quasi-isometric invariants, see \cite[p.\ 339]{Cornulier:qihlc}. This rules out the usage of Theorem \ref{thm:computable-reachability-sets2}.(i).
	
	The only non-trivial reachability set is the reachability set for Hausdorff dimension~1, and it is the same subspace \( \braket{X_1,X_2,X_3,X_5} \) for both. Its normaliser contains in addition \( X_6 \) in both cases, and not \( X_4 \), the next repeated normaliser being the full space. This rules out the usage of Theorem \ref{thm:computable-reachability-sets2}.(ii).
\end{Esim}

%%%%%%%%%%%%%%%%%%%%%%%%%%%%%%%
\section{On isometries of semi-direct products}

\label{sec:semidirect}

In this section we focus in proving Theorem \ref{thm:alkup_yleinen}. We restate it for the reader's convenience in a longer form.

\begin{lausei}[Theorem \ref{thm:alkup_yleinen}] %\label{thm:alkup_yleinen}
	Let \( H \) be a simply connected Lie group and \( \alpha \) a derivation of \( H\). Let \( \alpha = \ARe + \AIm + \Anil \) be the decomposition to real, imaginary and nilpotent parts as in Proposition \ref{prop:decomposition_of_a_derivation}, and denote \( \aaa=\ARe + \Anil \)
	Then the Lie group \(  H \rtimes_\alpha \mathbb{R} \) can be made isometric to the Lie group \(  H \rtimes_{\aaa} \mathbb{R} \).
\end{lausei}

%\begin{huom}
%	If \( H \) is assumed to be solvable, Theorem \ref{thm:alkup_yleinen} is a straightforward consequence of Fact \ref{thm:graaf-CKLNO} since the subspace \( \mathfrak{a} \) in the construction of the real-shadow (see Lemma \ref{eq:choosing-a-correctly}) may be chosen as \( \mathfrak{a} = \{0\} \times \mathbb{R} \subset \mathfrak{h} \rtimes_{A} \mathbb{R} \) and the modified bracket of Proposition \ref{prop:construction_of_the_real_shadow} then exactly agrees with the bracket of the semi-direct product \(  H \rtimes_{\aaa} \mathbb{R} \).
%\end{huom}

While Theorem \ref{thm:alkup_yleinen} may be applied outside the family of solvable groups, also within the family of solvable groups sometimes it might be practical to directly apply Theorem \ref{thm:alkup_yleinen} to find isometries between two solvable groups when neither of them is completely solvable.  Indeed, we remark that \(  H \rtimes_{\aaa} \mathbb{R} \) does not need to be completely solvable when \( H \) is not completely solvable. This approach would avoid the work to find their common real-shadow as in Fact \ref{thm:graaf-CKLNO}.

\begin{proof}[Proof of Theorem \ref{thm:alkup_yleinen}]
	The groups \( H \rtimes_\alpha \mathbb{R} \) and \( H \rtimes_{\aaa} \mathbb{R} \) may be seen as acting by left-translations on the manifold \( H \times \mathbb{R} \).
	Hence, the statement is proven by constructing a Riemannian metric on the manifold \( H \times \mathbb{R} \) for which both these actions are by isometries. 
	Denoting by \( 1\) the element \( (1_H,0) \in H \times \mathbb{R}  \),
	it is enough to construct a scalar product \( \rho \) on the tangent space \( T_{1} (H \times \mathbb{R}) \) with the following property% 
	\newcommand{\propJ}{(J)}
	\propJ
	\begin{itemize}
		\item[\propJ] whenever two elements \( g_1 \in H \rtimes_\alpha \mathbb{R} \) and \( g_2 \in H \rtimes_{\aaa} \mathbb{R} \) satisfy 
		%\( (g_1 \circ g_2)(1) = 1 \), 
		\( g_1(1) = g_2(1) \),
		then the differential of \( g_1^{-1} \circ g_2 \) is an isometry of the scalar product \( \rho \)
	\end{itemize}
	%so that . 
	If \( \rho \) satisfies the property \propJ, then it can be transported to a Riemannian metric on \(  H \times \mathbb{R}  \) with the desired properties.
	%For simplicity of the notation, we denote by \( 1 \) the element \( (1_H,0) \in H \times \mathbb{R} \). 
	
	For a derivation \( B \) on the Lie algebra of \( H \), the map \( \eta_t^B \) will denote the automorphism with the differential \( \e^{t B} \) for \( t \in \mathbb{R} \). This automorphism is well defined and unique since \( H \) is assumed to be simply connected, and we have \( \eta_t^B \circ \eta_s^B = \eta_{t+s}^B \). 
	Remark that, since \( \AIm \) is a semisimple map with purely imaginary eigenvalues, the subgroup \( W = \{ \eta^{\AIm}_t \times \Id \colon t \in \mathbb{R} \} \subset \pp{Aut}(H) \times \{\Id\} \) is precompact. Thus we may choose a scalar product \( \rho \) on \( T_{1} (H \times \mathbb{R}) \) that is invariant under (the differentials of) the maps in the closure of \( W \). We will next see that \( \rho \) has the property \propJ, thus finishing the proof.
	%It is then only left to show that the differential of \( g_1 \circ g_2 \) is an isometry of the scalar product \( \rho \) under the assumption that \( (g_1 \circ g_2)(1) = 1 \).
	
	%In the level of the groups, the group law of \( H \rtimes_\alpha \mathbb{R} \) is
	A point \( (h,t) \in H \times \mathbb{R} \) acts by left-translations with respect to the group law of \( H \rtimes_\alpha \mathbb{R} \) on the manifold \( H \times \mathbb{R}\) as
	\begin{equation} \label{eq:tulon_maar1}
	%(h,t)*(m,s) = (h * \eta_t^\alpha(m),t+s)
	L^\alpha_{(h,t)}(m,s) = (h * \eta_t^\alpha(m),t+s)
	\end{equation}
	and similarly for \( H \rtimes_{\aaa} \mathbb{R} \) by replacing \( \alpha \) with \( \aaa \). We deduce that if \( L_{(h,t)}^{\alpha}(1) = L_{(h',t')}^{\aaa}(1) \), then \( (h,t) = (h',t') \).
	%Thus, the left-translations by \( (h,t) \) w.r.t.\ the different group laws of \( H \rtimes_\alpha \mathbb{R} \) and \( H \rtimes_\aaa \mathbb{R} \) both map \( 1 \) to \( (h,t) \). Consequently, as we are only interested on two elements \( g_1 \in H \rtimes_\alpha \mathbb{R} \) and \( g_2 \in H \rtimes_{\aaa} \mathbb{R} \) that satisfy \( g_1(1) = g_2(1) \), then seeing \( g_1 \) and \( g_2 \) as points of the manifold \( H  \)  we actually have \( g_1=g_2 \) and we may denote this element by \( (h,t) \).
	%	We consider arbitrary points \( g_1 = (h_1,t_1)  \) and \( g_2 = (h_2,t_2) \) of \( H \times \mathbb{R} \) and the composition of the left-translations with respect to the two possible group laws, i.e., 
	%Now, consider the map
	Therefore, to establish the property \propJ, it is enough to show that the differential of the map
	\begin{equation*}
	Q_{(h,t)} = (L_{(h,t)}^{\alpha})^{-1} \circ L_{(h,t)}^{\aaa} %\piste
	\end{equation*}
	is an isometry for the scalar product \( \rho \) for every \( (h,t) \).
	%where \( L^B_{(h,t)} \) denotes the left-translation by \( (h,t) \) in the group law of \( H \rtimes_B \mathbb{R} \), for both \( B \in \{\alpha,\aaa\} \).
	By a straightforward computation one now finds
	\begin{equation*}
	Q_{(h,t)}(m,s) 
	= (\eta^{\alpha}_{-t}(\eta^{\aaa}_{t}(m)),s)  = 
	(\eta^{\AIm}_{-t}(m),s) \piste
	\end{equation*}
	This formula means that \( Q_{(h,t)} \in W \), and since \( \rho \) was chosen to be invariant under \( W \), we are done.
\end{proof}

We get the following corollary when combining Theorem \ref{thm:alkup_yleinen} and Fact \ref{thm:graaf-CKLNO}.
\begin{seur} \label{prop:R-shadow-and-real-part-reduction}
	Let \( \mathfrak{g} \) be a Lie algebra 
	%with a completely solvable subalgebra \( \mathfrak{h} \) of codimension 1.
	of the form \( \mathfrak{g} = \mathfrak{h} \rtimes_\alpha \mathbb{R} \), where \( \mathfrak{h} \) is completely solvable. Let \( \alpha = \ARe + \AIm + \Anil \) be the decomposition of \( \alpha \) as in Proposition \ref{prop:decomposition_of_a_derivation}. Then for \( \aaa =  \ARe + \Anil\) it holds that the Lie algebra \( \mathfrak{h} \rtimes_\aaa \mathbb{R} \) is the real-shadow of 
	\( \mathfrak{g} \).
\end{seur}

%%%%%%%%%%%%%%%%%%%%%%%%%%%%%%%%%%%%%%%
\section{Dimension 4}

%\subsection{Dimension 4}

\label{sec:4D}

The aim of this section is to prove Theorem \ref{thm:4D-classification_intro}. Namely, we find all pairs of  solvable simply connected 4-dimensional Lie groups that can be made isometric. We start from dimension 4 because dimensions 3 and below are already solved
%\footnote{Seba commented this, but maybe this was already changed by Enrico} 
(for a survey, see \cite{avain:fasslerledonne}).  
%By the Lie correspondence it is enough to consider all the finite-dimensional real solvable Lie algebras of dimension 4 up to isomorphisms. We will use these points of view, of groups and of algebras interchangeably. 

\subsection{Solvable groups up to isometry}
The isomorphism classes of all simply connected solvable Lie groups are known in dimension 4. Thus determining within this family the pairs of non-isomorphic groups that can be made isometric reduces by
Fact \ref{thm:graaf-CKLNO} to the determination of real-shadows of those solvable groups that are not completely solvable. 
Recall that by Fact \ref{thm:graaf-CKLNO} the relation \laina{can be made isometric} is transitive, and hence \emph{isometry (equivalence) classes} of groups are well defined objects.

The classification of simply connected Lie groups is equivalent to the classification of finite dimensional Lie algebras over \( \mathbb{R} \).
The list we shall use is given by Patera et al.\ \cite[Table I, p. 988]{avain:Patera}, which in turn is based on the classification of Mubarakzjanov \cite{avain:mubarakzjanov2,avain:mubarakzjanov1}. The list only contains the Lie algebras that are not direct products from lower dimension, and they are indexed from \( A_{4,1} \) to \( A_{4,12} \) with possible superscripts indicating one-parameter families. 
Table I in \cite{avain:Patera} also contains the classification of 3D Lie algebras; in what follows we shall use those names from \( A_{3,1} \) to \( A_{3,9} \) together with \( \mathbb{R}^n \) denoting the \(n\)-dimensional Abelian Lie algebra and \( A_2 \) denoting the unique non-Abelian 2D Lie algebra: the Lie group corresponding to \( A_2 \) was denoted by \( \pp{Aff}^+(\mathbb{R}) \) earlier.

In Table \ref{4D-luokittelu-real} we list all the simply connected completely solvable Lie groups of dimension 4. 
%Though this part is Already by Gordon and Wilson \cite[Theorems 4.3 and 5.2]{avain:GW88}, 
None of them can be made isometric to any other, and all the non-completely solvable groups (which in turn are listed in Table~\ref{4D-luokittelu-complex}) can be made isometric to exactly one of these.
In order to be able to divide the groups into families that seem to suit the purpose of classification up to isometry and quasi-isometry the best, in  Table~\ref{4D-luokittelu-real} we have relabelled the families in the left-most column, and we have written the unique completely solvable representative of the isometry class, in the notation of Patera et al., to the 2nd column. So the two left-most columns of Table~\ref{4D-luokittelu-real} serve as a dictionary. We indicate the range of parameters immediately after the labels. Two concrete examples on how to read the table: our label \( (2) \) denotes the Lie algebra \( \mathbb{R} \times  A_{3,1}  \) which is the direct product of the one-dimensional Abelian group and the Heisenberg group; instead, Lie algebra \( (6,\frac{1}{2},1)\) denotes the Lie algebra \( A_{4,5}^{1/2,1} \) of the classification of \cite{avain:Patera}. 

The right-most column of Table \ref{4D-luokittelu-real} has a mark X if and only if the isometry class of this group consists of more than one isomorphism classes of simply connected solvable Lie groups. 
The third column is about quasi-isometric classification, and we come back to it in Section \ref{sec:QI-classification-dim-4}.

\begin{table}[htb]
	\begin{tabular}{llll}
		\textbf{Our labelling} & \cite{avain:Patera} & \textbf{QI-type} & \(*\)
		\\
		\hline
		\((1)\)	&	\( \mathbb{R}^4 \)  & poly growth &X \YT \AT
		\\ \hline
		\((2)\)		&		\( \mathbb{R} \times A_{3,1} \)  & poly growth &X \YT \AT
		\\ \hline
		\((3)\)&				\( A_{4,1} \) & poly growth &\YT \AT
		\\ \hline
		\((4,a) \:\: a \in \vali{0,\infty}\)	&		\( A_{4,2}^a \)  & Heintze &\YT \AT
		\\ \hline
		\((5)\)		&	\( A_{4,4} \)   & Heintze &\YT \AT
		\\ \hline
		\((6,a,b) \:\: a,b \in \valil{0,1}, \: b > a\) &\( A_{4,5}^{a,b}  \)  & Heintze &\YT \AT
		\\ \hline
		\( (7,a) \:\: a \in \valil{0,1} \) &\( A_{4,5}^{a,a}  \)  & Heintze & X \YT \AT
		\\ \hline
		\((8)	\)	&	\( A_{4,7} \)   & Heintze &\YT \AT
		\\ \hline
		\((9,a)	\:\: a \in \vali{0,1} \)	&	\( A_{4,9}^{a}  \)  & Heintze &\YT \AT
		\\ \hline
		\((10)\)	&		\( A_{4,9}^1 \)   & Heintze & X \YT \AT
		\\ \hline
		\((11,a) \:\: a \in \vali{-\infty,0}\)	&		\( A_{4,2}^a \) & conedim 1 &\YT \AT
		\\ \hline
		\( (12,a,b) \:\: a,b \in \vali{-1,1} \pois \{0\}, \: b > a, \: a < 0  \) & \( A_{4,5}^{a,b}  \)  & conedim 1 &\YT \AT
		\\ \hline
		\((13,a)  \:\: a \in \valir{-1,0}  \) &\( A_{4,5}^{a,a}  \)  & conedim 1 & X \YT \AT
		\\ \hline
		\( (14,a)\:\: a \in \vali{-1,0} \) & \(A_{4,9}^{a} \quad  \)  & conedim 1 & \YT \AT
		\\ \hline
		\( (15) \)& \( A_{4,8} \)    & conedim 1 &\YT \AT
		\\ \hline
		\( (16)\) & \( A_{4,9}^0 \)    & conedim 2 &\YT \AT
		\\ \hline
		\( (17)\) & \( \mathbb{R} \times  A_{3,2} \)    & conedim 2 &\YT \AT
		\\ \hline
		\( (18)\) & \( \mathbb{R} \times  A_{3,3} \)    & conedim 2 & X \YT \AT
		\\ \hline
		\( (19)\) & \( A_2 \times  A_{2} \)    & conedim 2 & \YT \AT
		\\ \hline
		\( (20,a) \:\: a \in \vali{-1,1} \pois \{0\}\) & \( \mathbb{R} \times  A_{3,5}^a \)    & conedim 2 & \YT \AT
		\\ \hline
		\( (21)\) & \( \mathbb{R} \times  A_{3,4} \)    & conedim 2 & \YT \AT
		\\ \hline
		\( (22)\) & \( A_{4,3} \)    & conedim 3 & \YT \AT
		\\ \hline
		\( (23)\) & \( \mathbb{R}^2 \times A_2 \)    & conedim 3 & X \YT \AT
		\\ \hline \\
	\end{tabular}
	\caption{Completely solvable Lie algebras of dimension 4.}
	\label{4D-luokittelu-real}
\end{table}

Table \ref{4D-luokittelu-complex} lists all the remaining solvable Lie algebras of dimension 4. Namely, it lists those Lie algebras that are not completely solvable. Each of them has some completely solvable representative in its isometry class, namely the real-shadow. This real-shadow is indicated on the middle column, and our label for its isometry class is written in the right-most column. The computation of the real shadow is very simple after Corollary \ref{prop:R-shadow-and-real-part-reduction}.

\begin{table}[htb]
	\begin{tabular}{lll}
		\textbf{Lie algebra} & \textbf{real-shadow} & \textbf{isometry class}
		\\
		\hline
		\( \mathbb{R} \times A_{3,6} \) & \(\mathbb{R}^4\)  & \((1)\) \AT \YT
		\\ \hline
		\( A_{4,10} \) & \( \mathbb{R} \times A_{3,1} \) &  \( (2)\) \AT \YT
		\\ \hline
		\( A_{4,6}^{a,b} \:\: a,b \in \vali{0,\infty},\: a \le b \) & \( A_{4,5}^{a/b,a/b} \) & \( (7,a/b) \) \AT \YT
		\\ \hline
		\( A_{4,6}^{a,b} \:\: a,b \in \vali{0,\infty},\: a > b \) & \( A_{4,5}^{b/a,b/a} \) & \( (7,b/a) \) \AT \YT
		\\ \hline
		\( A_{4,11}^{a} \:\: a \in \vali{0,\infty} \) & \( A_{4,9}^1 \)  &  \( (10) \) \AT \YT
		\\ \hline
		\( A_{4,6}^{-a,b} \:\: a,b \in \vali{0,\infty},\: a\le b \) & \( A_{4,5}^{-a/b,-a/b} \) & \( (13,-a/b) \) \AT \YT
		\\ \hline
		\( A_{4,6}^{a,-b} \:\: a,b \in \vali{0,\infty},\: b < a  \) & \( A_{4,5}^{-b/a,-b/a} \) & \( (13,-b/a) \) \AT \YT
		\\ \hline
		\( \mathbb{R} \times A_{3,7}^a \:\:  a \in \vali{0,\infty} \)  & \(\mathbb{R} \times A_{3,3} \) & \((18)\)\AT \YT
		\\ \hline
		\( A_{4,12} \)  & \(\mathbb{R} \times A_{3,3} \) & \((18)\)\AT \YT
		\\ \hline
		\( A_{4,6}^{a,0} \)  & \(\mathbb{R}^2 \times  A_2\)  & \( (23) \) \AT \YT
		\\ \hline \\
	\end{tabular}
	\caption{Solvable but not completely solvable Lie algebras of dimension 4, and their real-shadows.}
	\label{4D-luokittelu-complex} 
\end{table}

\subsection{Dropping the assumption of solvability}
	We do not have many tools to treat non-solvable simply connected Lie groups. However, in dimension 4 there are no Levi decompositions other than the direct products (see \cite[p.\ 301]{MacCallum1999}), and hence the only two non-solvable Lie groups are \( \mathbb{R} \times  \mathbb{S}^3 \) and \( \mathbb{R} \times \widetilde{\pp{SL}}(2) \), and we can say something about these.
	We are thus interested if either of these two groups can be made isometric to some solvable groups, or if they can be made isometric to each other. For topological reasons, \( \mathbb{R} \times  \mathbb{S}^3 \) cannot be made isometric to any other simply connected 4-dimensional group: it is the only group not homeomorphic to \( \mathbb{R}^4 \).
	%quasi-isometric to \( \mathbb{R} \), unlike any other Lie group of dimension 4, so it cannot be isometric to any other group. 
	The case of the group \( \mathbb{R} \times \widetilde{\pp{SL}}(2) \) however is more involved, and we will see next what we can say about it.
	
	We know from Proposition \ref{prop:CKLNO-selvennys2} that \(\widetilde{\pp{SL}}(2) \) can be made isometric to \( \mathbb{R} \times A_2 \), hence the groups \( \mathbb{R} \times \widetilde{\pp{SL}}(2) \) and \( \mathbb{R}^2 \times A_2 \) can be made isometric. Consequently, \( \mathbb{R} \times \widetilde{\pp{SL}}(2) \) must have cone dimension 3, which is the cone dimension of \( \mathbb{R}^2 \times A_2 \) as one may see from Proposition \ref{prop:cone-dimensions-algebraically}.
	Thus, checking the cone dimensions of solvable groups from Table \ref{4D-luokittelu-real}, only the question remains
	whether or not  \( \mathbb{R} \times \widetilde{\pp{SL}}(2) \) can be made isometric also to the group \( A_{4,3} \) or to some groups in the family \( A_{4,6}^{a,0} \) for \( a \in \mathbb{R} \).
	%whose Lie algebra \( = \braket{X,Y,Z,T} \) with bracket relations \( [T,X] = X \) and \( [T,Z] = Y \). 
	%Remark that since the transitivity of the isometry-relation is not established without the assumption solvability, the possible isometry between  \( \mathbb{R} \times \widetilde{\pp{SL}}(2) \) and \( A_{4,3} \) is not ruled out by 
	Notice for example that 
	the fact that \( A_{4,3} \) and \( \mathbb{R}^2 \times A_2 \) cannot be made isometric
	does not rule out that \( \mathbb{R} \times \widetilde{\pp{SL}}(2) \) and \( A_{4,3} \) can be made isometric, because \( \mathbb{R} \times \widetilde{\pp{SL}}(2) \) is not solvable.
	Similarly, if \( \mathbb{R} \times \widetilde{\pp{SL}}(2) \) can be made isometric to \( A_{4,6}^{a,0} \) for some \( a \in \mathbb{R} \), it does not imply anything for \( A_{4,6}^{a',0} \) with \( a' \neq a \). One might wish to compare this phenomenon to \cite[Theorem 4.21]{avain:CKLNO}.

\subsection{Quasi-isometric classification of 4-dimensional groups}

\label{sec:QI-classification-dim-4}

We don't have a complete quasi-isometric classification of simply connected 4-dimensional Lie groups. In this section we show what is known about it. Recall that quasi-isometry equivalence classes are necessarily unions of the isometry classes, and these isometry classes we just established for simply connected solvable groups. Hence it is enough to consider the groups in Table \ref{4D-luokittelu-real} and the two non-solvable groups \( \mathbb{R} \times  \mathbb{S}^3 \) and \( \mathbb{R} \times \widetilde{\pp{SL}}(2) \).

%We may start from \( \mathbb{R} \times  \mathbb{S}^3 \). Recall that by a result of J. Roe \cite{roe}, the topological dimension is a quasi-isometry invariant for completely solvable groups. Hence 
Recall that the degree of polynomial growth is a quasi-isometric invariant. 
The degree of polynomial growth for \( \mathbb{R} \times  \mathbb{S}^3 \) is 1, so it cannot be quasi-isometric either to any group in Table \ref{4D-luokittelu-real} or the group \( \mathbb{R} \times \widetilde{\pp{SL}}(2) \).
%, are quasi-isometric to \( \mathbb{R} \), and the latter in turn is quasi-isometric to \( \mathbb{R} \times  \mathbb{S}^3 \). 
Consequently, the quasi-isometry class of \( \mathbb{R} \times  \mathbb{S}^3 \) within the family of simply connected 4-dimensional groups, is a singleton.

About the group \( \mathbb{R} \times \widetilde{\pp{SL}}(2) \), the only thing that we are able to say is that since it can be made isometric to \( \mathbb{R}^2 \times A_2 \), then it must have cone dimension 3. 

For all completely solvable groups that are not Heintze groups and do not have polynomial growth, we have calculated, using Proposition \ref{prop:cone-dimensions-algebraically}, their cone dimensions and marked them to
the third column titled \laina{QI-type} of Table \ref{4D-luokittelu-real}. The cone dimensions are quasi-isometry invariants by \cite{MR3006687}. 

Recall that while Heintze groups have cone dimension 1, they are quasi-isometrically distinct from the non-Heintze groups of cone dimension 1 since in dimension 4 only the Heintze groups are Gromov hyperbolic  by \cite{MR2826945} (see also \cite[p.\ 277]{Cornulier:qihlc}).

The quasi-isometric classification of 4-dimensional purely real Heintze groups can be done by case-by-case study.
However, a direct argument follows from Theorem~\ref{thm:computable-reachability-sets2}.(i), Proposition \ref{go-to-the-boundary} and the results of Xie \cite{MR3180486}, Carrasco Piaggio and Sequeira \cite[Theorem 1.3]{avain:CPS}: The purely real Heintze groups in Table \ref{4D-luokittelu-real} split into two categories
\begin{align*}
&\text{nilradical } \mathbb{R}^3 \qquad (4,a) \quad (5) \quad (6,a,b)
\\
&\text{nilradical } \pp{Heis} \qquad (8) \quad (9,a) \quad (10)
\end{align*}
Those with nilradical \( \mathbb{R}^3 \) are quasi-isometrically distinct from each other by \cite{MR3180486}. Those with nilradical \( \pp{Heis} \) are quasi-isometrically distinct from each other by \cite[Theorem 1.3]{avain:CPS}. All the quasi-isometry relations between these two classes are excluded by Theorem~\ref{thm:computable-reachability-sets2}.(i). Thus, the quasi-isometry classes, isometry classes, and isomorphism classes all agree for purely real Heintze groups of dimension 4.

For the groups of polynomial growth,
our classes (1), (2) and (3) are known to be quasi-isometry equivalence classes, because the completely solvable representatives (in this case, nilpotent representatives) are Carnot groups and quasi-isometric classification of Carnot groups is solved by Pansu \cite{pansu}. 

%In conclusion, the isometry classes with labels from (1) to (10) agree with the quasi-isometry classes, but for the isometry classes with labels (11) onwards, not much else is known except the cone dimensions.
As a conclusion, we may present the following proposition.
\begin{prop}
	Let \( \mathcal{G} \) be the family of the isomorphism classes of 4-dimensional simply connected solvable groups that either have polynomial growth or are Heintze groups. Then two elements \( G,H \in   \mathcal{G} \) are quasi-isometric if and only if they can be made isometric. If the groups \( G \) and \( H \) are completely solvable, then they are quasi-isometric if and only if they are isomorphic.
\end{prop}

%%%%%%%%%%%%%%%%%%%%%%%%%%%%%%%%%%%%%%%%%
\section{Dimension 5}

\label{sec:5D}

The classification of real solvable Lie algebras is known in dimensions five also, see \cite{avain:Patera}. However, due to the multitude of isomorphism classes, we rather restrict our attention to the groups of polynomial growth.
%, as they play an important role from the viewpoint of nilpotent groups. Also, calculating the real-shadow for the groups of polynomial growth is simpler, because it goes back to a somewhat simpler concept called \laina{nilshadow}. 

\begin{table}[htb]
	\begin{tabular}{llll}
		\textbf{Patera et al.} & \textbf{de Graaf} &  \textbf{nilshadow} & \( G_\infty \) 
		\\	\hline
		\( \mathbb{R}^5 \) &   &  &  \AT \YT
		\\	\hline
		\( \mathbb{R}^2 \times A_{3,1} \) &   &  &  \AT \YT
		\\	\hline
		\( \mathbb{R} \times A_{4,1} \) &   &  &  \AT \YT
		\\	\hline
		\( A_{5,1} \) &  \(L_{5,8}\) &  &  \AT \YT
		\\	\hline
		\( A_{5,2} \) & \( L_{5,7} \)  &  &  \AT \YT
		\\	\hline
		\( A_{5,3} \) &  \(L_{5,9}\) &  &  \AT \YT
		\\	\hline
		\( A_{5,4} \) & \( L_{5,4} \)  &  &  \AT \YT
		\\	\hline
		\( A_{5,5} \) &  \(L_{5,5}\) &  & \( \mathbb{R} \times A_{4,1} \)  \AT \YT
		\\	\hline
		\( A_{5,6} \) & \( L_{5,6} \)  &  & \( A_{5,2} \) \AT \YT
		\\	\hline
		\( \mathbb{R}^2 \times A_{3,6} \) &  & \( \mathbb{R}^5 \) &  \AT \YT
		\\	\hline
		\( \mathbb{R} \times  A_{4,10} \) &  & \( \mathbb{R}^2 \times A_{3,1} \) &  \AT \YT
		\\	\hline
		\( A_{5,17}^{s,0,0} \: \: s \neq 0 \) &  & \( \mathbb{R}^5  \) &  \AT \YT
		\\	\hline
		\( A_{5,14}^0 \) &  & \( \mathbb{R}^2 \times A_{3,1} \) &  \AT \YT
		\\	\hline
		\( A_{5,26}^{0, \varepsilon} \:\: \varepsilon= \pm 1 \) &  & \( A_{5,4} \) &  \AT \YT
		\\	\hline
		\( A_{5,18}^{0} \) &  & \( A_{5,1} \) &  \AT \YT
		\\ \hline \\
	\end{tabular}
	\caption{Solvable Lie algebras of type (R) in dimension 5.}
	\label{table:5d-poly-growth}
\end{table}

The first task is to determine a list of all simply connected solvable Lie groups of polynomial growth in dimension 5. We are not aware of a reference where this is done, so we have to do it by ourselves using the classification of real solvable Lie algebras presented in Patera et al.\ \cite[p. 989]{avain:Patera}. Notice that one can pretty quickly find all the candidates for groups of polynomial growth by excluding the Lie algebras with a bracket relation of the type \( [e_i,e_j] = \lambda e_j \) for \( \lambda \neq 0 \): This is an obstruction of being polynomial growth, since all the eigenvalues of all the adjoint maps should be purely imaginary (see Section \ref{sec:preliminaries_isometries}). The candidates so found are possible to check by hand if they have polynomial growth or not.

Taking into account the direct products, the full list of solvable simply connected Lie groups of polynomial growth is presented in Table \ref{table:5d-poly-growth}. In the first 2 columns, we have recalled a dictionary between classifications presented in Patera et al.\ and that by de Graaf \cite{deGraaf-2007-dim_6_nilpotent_lie_algebras} for nilpotent Lie algebras. For nilpotent algebras that are not Carnot algebras, we have indicated their associated Carnot algebras in the 4th column. For non-nilpotent Lie algebras, we have indicated their nilshadow in the 3rd column.

From the algebraic classification given in Table \ref{table:5d-poly-growth} one may directly deduce the classification up to isometries and quasi-isometries (up to one open case we will mention soon) using the list of invariants we recorded in beginning of Section \ref{sec:algebraic_tools_for_QI}. Indeed, recalling invariant \ref{inv:D} and Remark \ref{rmk:nilshadow_is_real}, every group is isometric to its nilshadow, and in dimension 4 it happens that the nilshadows are always Carnot groups (Carnot groups are those with empty field both in \laina{nilshadow} and in \laina{\(G_\infty\)}). Moreover, the nilshadows happen to be those Carnot groups that are not associated Carnot groups of some nilpotent non-Carnot groups. Hence the classification up to isometry and quasi-isometry is ready for the groups of polynomial growth and those Carnot groups that appear as their nilshadows. Only problem that remains after applying \ref{inv:A} is if \( A_{5,5} \) or \( A_{5,6} \) are quasi-isometric to their associated Carnot groups, recall that by \cite{avain:Wol63} isometries between non-isomorphic nilpotent groups cannot exist. The invariant \ref{inv:B} tells that \( A_{5,5} \) is not quasi-isometric to its associated Carnot group \( \mathbb{R} \times A_{4,1} \) (see \cite[Section 19.7]{Cornulier:qihlc}), but the possible quasi-isometry relation between \( A_{5,6} \) and \( A_{5,2} \) remains unanswered by this analysis.%, and it is an open problem of the field indeed.

In conclusion, as was the case for the family of simply connected solvable Lie groups of dimension 4, we are unable to completely classify simply connected Lie groups of polynomial growth in dimension 5. However, here it is only one pair of groups whose possible quasi-isometry relation remains open: whether or not the Lie group \( A_{5,6} \) is quasi-isometric to its associated Carnot group \( A_{5,2} \). This question cannot be answered by the community for now.

%%%%%%%%%%%%%%%%%%%%%%%%%%%%%%%%%%%%%%%
%%%%%%%%%%%%%%%%%%%%%%%%%%%%%%%%%%%%%%%
%\appendix
%%%%%%%%%%%%%%%%%%%%%%%%%%%%%%%%%%%%%%%
%%%%%%%%%%%%%%%%%%%%%%%%%%%%%%%%%%%%%%%

\section{Appendix: A direct proof in Abelian case}

%\subsection{Abelian case}
\label{sec:abelian-theorem}

In this section we prove Theorem~\ref{lause:abelianity_preserves_with_Seba}. It is a less general statement than Theorem~\ref{thm:computable-reachability-sets2}.(i), but the proof is completely different in spirit and might have independent interest and possibilities to generalise. The proof is highly inspired by the results of \cite{avain:CPS}.

The following definition appeared implicitly in \cite{avain:CPS}, but we prefer to have a name for it.
% (for the terminology, see Section \ref{sec:subsec:prelim-heintze} in particular).

\begin{maar} \label{maar:characteristic_subgroup}
	The \emph{characteristic subalgebra} for  a purely real homogeneous group \( (N,\alpha) \) is the subalgebra \( \mathfrak{h}_\alpha \) of \( \mathfrak{n} \) constructed as follows. 
	Consider a basis of \( \mathfrak{n} \) where \( \alpha \) is in Jordan form, and 
	let \( \lambda_1 \) denote the smallest of the eigenvalues of \( \alpha \).
	Let \( V_{\lambda_1} \) be the subspace corresponding to the Jordan-blocks of \( \alpha\) of eigenvalue \( \lambda_1 \)  (i.e., the generalised eigenspace of eigenvalue \( \lambda_1\)). Let \( \hat{V}_1 \subset V_{\lambda_1} \) be the sum of the subspaces corresponding to the Jordan blocks in \( V_{\lambda_1} \) of maximal size. Next, let \( \mathcal{V}_1 \) consist of eigenvectors of eigenvalue \( \lambda_1 \)  inside \( \hat{V}_1 \), and finally define \( \mathfrak{h}_\alpha = \pp{LieSpan}(\mathcal{V}_1) \). 
	We further denote by \( H_\alpha \) the subgroup of \( N \) with Lie algebra \( \mathfrak{h}_\alpha \) and call it the \emph{characteristic subgroup} of \( (N,\alpha) \).
\end{maar}

\begin{huom}
	Equivalently, the characteristic subalgebra is defined as follows: Let \( k \in \mathbb{N} \) be the unique integer such that \( (\alpha|_{V_{\lambda_1}} - \lambda_1 \Id)^k \neq 0 \) and \( (\alpha|_{V_{\lambda_1}} - \lambda_1 \Id)^{k+1} = 0 \). Then \( \mathcal{V}_1 = \pp{Im}(\alpha|_{V_{\lambda_1}} - \lambda_1 \Id)^k\) and \( \mathfrak{h}_\alpha = \pp{LieSpan}(\mathcal{V}_1) \).
\end{huom}

In the following, we list some facts related to characteristic subalgebras and subgroups.
\begin{prop}\label{prop:derivations_preserve_char-subalg}
	\begin{lista}
		\item
		\(\mathfrak{h}_\alpha = \mathfrak{n}  \) if and only if \( (N,\alpha) \) is of Carnot type. 
		\item	\( \mathfrak{h}_\alpha \) is preserved under \( \alpha \).
		\item	Suppose \( F \colon (N_1,\alpha) \to (N_2,\beta) \) is a biLipschitz map between two purely real homogeneous groups, and suppose \( F(1_{N_1}) = 1_{N_2} \). Then \( F(H_\alpha) = H_\beta \).
	\end{lista}
\end{prop}

\begin{proof}
	The part (i) follows by observing that both of the claims are equivalent to the condition \( V_{\lambda_1} = \mathcal{V}_1 \).
	
	The part (ii) is proven by a straightforward induction on the length of a bracket in~\( \mathfrak{h}_\alpha \), using that \( \mathcal{V}_1 \) is preserved under \( \alpha \) by construction.
	
	The part (iii) is proven in \cite{avain:carrasco-orlicz}, see also \cite[p.\ 6]{avain:CPS}
\end{proof}

\begin{lause} \label{lause:abelianity_preserves_with_Seba}
	Let 
	\( (N_1,\alpha) \) and \( (N_2,\beta) \) be purely real homogeneous groups that are biLipschitz equivalent.
	%	Let \( H_\alpha \) be the characteristic subgroup of \( (N_1,\rho_\alpha) \).
	%	%\( N_1 \rtimes_\alpha \mathbb{R} \) and \( N_2 \rtimes_\beta \mathbb{R} \) be purely real Heintze groups that are quasi-isometric.
	%	If \( H_\alpha \) is Abelian and \( H_\alpha \normal N_1 \), then it also holds \( H_\beta \normal N_2 \) and any quasi-symmetry \( F  \colon (N_1,\rho_\alpha) \to (N_2,\rho_\beta)  \) fixing the identity is an isomorphism of Lie groups.
	If \( N_1 \) is Abelian, so is \( N_2 \). Consequently \( (N_1,\alpha) \) and \( (N_2,\beta) \) are isomorphic as homogeneous groups, by \cite{MR3180486}. 
\end{lause}

\begin{proof}
	We prove the claim inductively on the topological dimension of the groups in question. 
	The case \( n=1 \) (and also \( n=2 \)) is true due to the lack of non-Abelian nilpotent groups. So assume the claim holds for groups of dimension \( k \) and less and 
	\begin{equation} \label{eq:formula_of_dimension}
	\dim(N_1) = \dim(N_2) =k+1 \piste
	\end{equation}
	
	Let \( F  \colon N_1 \to N_2  \) be a biLipschitz map, which after post-composing with a left-translation we may assume to satisfy \( F(1_{N_1}) = 1_{N_2} \). 
	Thus, when \( H_\alpha \) and \( H_\beta \) denote the respective characteristic subgroups,  by Proposition \ref{prop:derivations_preserve_char-subalg}.(iii) it holds 
	\begin{equation} \label{eq:Falphabeta}
	F ( H_\alpha) = H_\beta \piste
	\end{equation}
	If \(  H_\alpha = N_1 \), then by Proposition \ref{prop:derivations_preserve_char-subalg}.(i) 
	the homogeneous group \( (N_1, \alpha) \) is of Carnot type,
	and as a consequence of \cite[Theorem 1.9]{avain:carrasco-orlicz} the homogeneous group \( (N_2,\beta) \) is also of Carnot type. In this case, by Pansu's Theorem \cite{pansu}, \( (N_1, \alpha) \) and \( (N_2, \beta) \) are isomorphic as homogeneous groups.
	We are left to consider the case 
	\begin{equation} \label{eq:Halpanotequaln1}
	H_\alpha \subsetneq N_1 \piste
	\end{equation}

	From \eqref{eq:formula_of_dimension} and \eqref{eq:Halpanotequaln1} we have \( \dim(H_\alpha) \le k \).
	%Since the dimensions satisfy \eqref{eq:formula_of_dimension} and we have \( H_\alpha \subsetneq N_1 \) and \( F ( H_\alpha) = H_\beta \), the dimensions of characteristic subgroups are at most \( k \). 
	Thus the induction assumption and \eqref{eq:Falphabeta} gives that \(H_\beta \) is Abelian because \( H_\alpha \) is Abelian.
	
	Moreover we claim that \( H_\beta \) is normal in \( N_2 \). Indeed, the normaliser of \( H_\alpha \) is \( N_1 \) since \( N_1\) is Abelian, hence by	Proposition \ref{lemma:CPS-corollary-about-normalisers} the normaliser of \( H_\beta \) is \( N_2 \).

	Next, we claim \( H_\beta \) is central in \( N_2\).
	By the definition of the characteristic subalgebra
	%	By Definition \ref{maar:characteristic_subgroup} 
	we have \( \mathfrak{h}_\beta = \pp{LieSpan}(\mathcal{V}_1) \), where \( \mathcal{V}_1 \subset V_{\lambda_1} \), as in Definition \ref{maar:characteristic_subgroup}. We know now that \( \mathfrak{h}_\beta \) is an Abelian ideal, so \( \mathcal{V}_1 \) is Abelian and \( \mathfrak{h}_\beta \subset V_{\lambda_1} \). Using the grading given by the generalised eigenspaces \( V_\lambda \) of \( \beta \)  (see \cite[p. 16 Prop. 12]{bourbaki-lie-7-9-fr}) and the fact that \( \mathfrak{h}_\beta \) is an ideal of \( \mathfrak{n}_2 \)  we get for all \( H \in \mathfrak{h}_\beta \) and \( X \in \mathfrak{n}_2 \) that
	\[ [X,H] \in \mathfrak{h}_\beta \cap \bigoplus_{\lambda > \lambda_1} V_\lambda = \{0\} \piste \]
	Hence \( \mathfrak{h}_\beta \) is central in \( \mathfrak{n}_2 \).
	
	Take \( W \) to be a complementary subspace to \( \mathcal{V}_1 \) inside \( V_{\lambda_1} \). 
	Define
	%, using the grading given by the generalised eigenspaces \( V_\lambda \) of \( \beta \)  (see \cite[p. 16 Prop. 12]{bourbaki-lie-7-9-fr}),
	\[ \mathfrak{s}_{\beta} = W \oplus \bigoplus_{\lambda > \lambda_1} V_\lambda \pilkku \]
	which is an ideal because it contains \( [\mathfrak{n}_2,\mathfrak{n}_2] \). The subspaces \( \mathfrak{s}_{\beta} \) and \( \mathfrak{h}_{\beta} \) are in direct sum and they are both ideals, so the Lie algebra \( \mathfrak{n}_2 \) is the direct product of these two subalgebras: \( \mathfrak{n}_2 = \mathfrak{s}_{\beta} \times  \mathfrak{h}_{\beta} \). On the \( N_1 \) side, the same construction works but it is simpler because \( N_1 \) is Abelian. Anyway, we may decompose \( \mathfrak{n}_1 = \mathfrak{h}_\alpha \times \mathfrak{s}_\alpha \), where \( \mathfrak{s}_\alpha \) is an arbitrary
	%\footnote{answer to a question of Enrico: \( \alpha \) does not preserve \( \mathfrak{s}^\beta \) or \( \mathfrak{s}^\alpha \), basically since \( \alpha \) can have a nilpotent block and in this case \( \alpha \) tilts any complementary subspace to \( \mathcal{V}_1 \), in particular it tilts \( W \) in the definition of \( \mathfrak{s} \).  However, this is not a problem: We show that the quotients \( N_1/H_\alpha \) and \( N_2/H_\beta \) are biLip-eq homogeneous groups, and first of them is Abelian, so is the second by induction. Then from the direct product structure at the level of the groups, \( N_2/H_\beta \) is isomorphic as a group to \( S^\beta \), so the latter is Abelian. We could say something more on this to clarify.} 
	complementary subspace to \( \mathfrak{h}_\alpha \).

	By Proposition \ref{prop:derivations_preserve_char-subalg}.(ii) and the concrete formula for a homogeneous distance on the quotient given in Lemma \ref{huom:preserved-subalgebra-good-quotient}, we have that the quotient groups 
	%	We claim next that the quotient groups 
	\( N_1/H_\alpha \) and \( N_2/H_\beta \) are biLipschitz equivalent purely real homogeneous groups. Their dimension is at most \( k \), since the characteristic subgroups have at least dimension 1. 
	%Firstly,  the derivation \( \alpha \) induces a derivation on the quotient space \( N/H_\alpha \) because the characteristic subalgebras are preserved. We already noted that our biLipschitz map sends \( H_\alpha \) to \( H_\beta \), hence it induces a 
	Hence by induction, \( N_2/H_\beta \) is Abelian since \( N_1/H_\alpha \) is Abelian. 
	By the structure of direct products, \( N_2/H_\beta \) and \( S_\beta \) are isomorphic as Lie groups, hence \( S_\beta \) is Abelian.
	Since \( N_2 \) is a direct sum of two Abelian normal subgroups \( H_\beta \) and \( S_\beta \), then \( N_2 \) is Abelian.
	
	For the final statement, \cite[Theorem 1.1]{MR3180486} tells that the Jordan forms of \( \alpha \) and \( \beta\) are proportional. On the other hand, since the homogeneous groups are biLipschitz equivalent, the smallest of the eigenvalues of \( \alpha \) and \( \beta \) must agree since by Proposition \ref{prop:CPS-H-dim-lemma} we have that the common smallest eigenvalue \( \lambda_1 \) is the minimal Hausdorff dimension of curves. Therefore the Jordan forms of \( \alpha \) and \( \beta \) agree and this is enough to give an isomorphism of homogeneous groups in the Abelian case.
\end{proof}

%We use Theorem \ref{lause:abelianity_preserves_with_Seba} in Section \ref{sec:QI-classification-dim-4} to classify 
%%For a simple application, in Section \ref{sec:QI-classification-dim-4}, after recalling the algebraic classification of 
%4-dimensional Heintze groups
%%, we use Theorem \ref{lause:abelianity_preserves_with_Seba} to classify them 
%up to quasi-isometry.

\begin{huom}
	The proof above does not give a new proof of the main result of \cite{MR3180486}, since it may happen that the complementary subspace \( W \) cannot be chosen to be preserved under the derivation \( \beta \). Therefore, while \( N_2 / H_\beta \) has a structure of a homogeneous group induced by \(\beta \), the subgroup \( S^\beta \) is not preserved under \( \beta\) and does not inherit a structure of a homogeneous group.
\end{huom}

\begin{huom}
	Theorem \ref{lause:abelianity_preserves_with_Seba} may also be proven 
	%from Theorem \ref{thm:computable-reachability-sets2}.(ii) by an argument that we will next sketch, thereby giving a third proof for Theorem \ref{lause:abelianity_preserves_with_Seba}. A homogeneous group \( (N,\alpha) \) is non-Abelian if and only if for some \( s > 0 \) the reachability set \( (N,\alpha)^s \) is strictly larger than the subgroup corresponding to \( \bigoplus_{0 < \lambda \le s} V_\lambda \). On the other hand, if homogeneous groups \( (N_1,\alpha) \) and \( (N_2,\beta) \) are biLipschitz equivalent, then the characteristic polynomials of \( \alpha \) and \( \beta \) agree by \cite{avain:CPS}, and hence the dimensions of the generalised eigenspaces of the same eigenvalues agree. But a biLipschitz homeomorphism cannot map a submanifold to a submanifold of different dimension.
	from  Theorem \ref{thm:computable-reachability-sets2}.(ii) by an argument that we will next sketch, thereby giving a third proof for Theorem \ref{lause:abelianity_preserves_with_Seba}. A homogeneous group \( (N,\alpha) \) is non-Abelian if and only if for some \( s > 0 \) the reachability set \( (N,\alpha)^s \) is strictly larger than the subgroup corresponding to \( \bigoplus_{0 < \lambda \le s} V_\lambda \).
	Suppose that homogeneous groups \( (N_1,\alpha) \) and \( (N_2,\beta) \) are biLipschitz equivalent.
	On the one hand, the characteristic polynomials of \( \alpha \) and \( \beta \) agree by \cite{avain:CPS}, and hence the dimensions of the generalised eigenspaces of the same eigenvalues agree.
	On the other hand, \( (N_1,\alpha)^s \) and \( (N_2,\beta)^s \) have the same dimension for every \(s\) by Theorem \ref{thm:computable-reachability-sets2}.(ii).
	We conclude that, if \(N_1\) is Abelian, then also \(N_2\) is Abelian.
\end{huom}

%%%%%%%%%%%%%%%%%%%%%%%%%%%%%%%%% 
%%%%%%%%%%%%%%%%%%%%%%%%%%%%%%%%% 
\bibliography{viittaukset}
\bibliographystyle{amsalpha}
\end{document}